\documentclass[final,reqno]{elsarticle}
\usepackage{lineno}
\usepackage{stmaryrd}
\usepackage{mathrsfs}
\usepackage{amssymb,amsmath,amsfonts,latexsym}
\usepackage{amsthm}
\usepackage{graphicx,float,epsfig,color,fancyhdr}
\usepackage{colortbl}
\usepackage[colorlinks=false, linkcolor=blue,  citecolor=OliveGreen,  pdfstartview={}{}]{hyperref} 
\newtheorem{lemma}{Lemma}[section]
\newtheorem{remark}{Remark}[section]

\newtheorem{theorem}{Theorem}[section]

\newtheorem{problem}{Problem}

\usepackage[dvipsnames]{xcolor}

\newcommand{\GR}[1]{{\color{magenta}#1}}

\def\CE{\mathcal{E}}
\def\CM{\mathcal{X}}
\def\CN{\mathcal{Y}}

\def\CT{\mathcal{T}}
\def\E{\mathrm{K}}
\def\G{\Gamma}

\def\H{\mathrm{H}}

\def\HuO{\H^1(\O)}

\def\LO{L^2(\O)}

\def\div{\mathop{\mathrm{div}}\nolimits}
\def\bn{\boldsymbol{n}}
\def\bu{\boldsymbol{u}}
\def\bv{\boldsymbol{v}}

\def\bT{\mathbf{T}}

\def\O{\Omega}

\def\R{\mathbb{R}}

\def\hdel{\widehat{\delta}}

\def\l{\lambda}

\def\sp{\mathop{\mathrm{sp}}\nolimits}

\def\Vh{V_h}

\def\bsig{\boldsymbol{\sigma}}
\def\btau{\boldsymbol{\tau}}

\def\CT{\mathcal{T}}
\def\G{\Gamma}
\def\HrE{{\H^{r}(\E)}}

\def\HtO{{\H^{t}(\O)}}

\def\HuO{{\H^1(\O)}}

\def\HutO{{\H^{1+t}(\O)}}
\def\H{\mathrm{H}}
\def\L{\mathrm{L}}
\def\LO{\L^2(\O)}

\def\O{\Omega}
\def\R{{\mathbb{R}}}

\def\bS{\mathbf{S}}
\def\bT{T}
\def\bV{\mathcal{V}}
\def\mQ{\mathcal{Q}}

\def\bn{\boldsymbol{n}}

\def\bpi{\boldsymbol{\Pi}}

\def\bu{\boldsymbol{u}}

\def\bv{\boldsymbol{v}}

\def\div{\mathop{\mathrm{div}}\nolimits}

\def\ds{\,ds}

\def\hdel{\widehat{\delta}}
\def\l{\lambda}

\def\rot{\mathop{\mathrm{rot}}\nolimits}

\renewcommand\sp{\mathop{\mathrm{sp}}\nolimits}
\newcommand\bbP{\mathbb{P}}

\journal{}
\date{\today}

\def\rot{\text{rot}}

\begin{document}
\begin{frontmatter}

\title{A priori error analysis for a mixed VEM discretization
of the spectral problem for the Laplacian operator.}

\author[1]{Felipe Lepe}
\ead{flepe@ubiobio.cl}
\address[1]{Departamento de Matem\'atica,
Universidad del B\'io-B\'io, Casilla 5-C, Concepci\'on, Chile.}
%\address[2]{CI$^2$MA, Universidad de Concepci\'on, Casilla 160-C, Concepci\'on, Chile.}
\author[2]{Gonzalo Rivera}
\ead{gonzalo.rivera@ulagos.cl}
\address[2]{Departamento de Ciencias Exactas, Universidad de Los Lagos,
Casilla 933, Osorno, Chile.}
%\author[2]{Iv\'an Vel\'asquez}
%\ead{ivelasquez@ci2ma.udec.cl}

\begin{abstract} 
The aim of the present work is to derive a error estimates
for the Laplace eigenvalue problem in mixed form, by means of a virtual element method. With the aid of  the theory for non-compact operators, we prove that the proposed method is spurious free and convergent.
We prove optimal order  error estimates for the eigenvalues and eigenfunctions. Finally, we report numerical tests to confirm the theoretical results together with a rigorous computational analysis of the effects of the stabilization in the computation of the spectrum.
\end{abstract}


\end{frontmatter}

%-----------------------------------------------------------------------

\setcounter{equation}{0}
\section{Introduction}
\label{SEC:INTR}

In the recent years, the virtual element method (VEM), which is a generalization of the classic
finite element method to polygonal meshes, has shown important breakthroughs in the numerical
resolution of partial differential equations. 
%In particular, the VEM presents important results in eigenvalue problems.

The eigenvalue problems are a subject of study where the classic numerical methods provided by the finite element method (FEM)
has been plenty developed in different contexts, as for example, acoustic interactions, elastoacustic problems, elasticity problems, vibrations of structures, fluid stability, etc. Due the importance of knowing the natural vibration frequencies of the mentioned problems, is relevant to have numerical tools that improve the accuracy  in the approximation of solutions, with reduced computational costs. Is in this sense where the VEM presents important features in comparison with FEM, that makes it attractive for mathematicians and engineers.

A priori error estimates for spectral problems implementing VEM has been developed in the past years, with important results. We mention \cite{BMRR17, CGMMV, GMV,  GV,  LMRV, MM,MR2019, MRR1,  MV2}, only to mention a few.  Although, in the literature is possible to find several studies on the implementation of VEM methods for mixed formulations like, \cite{ABMV,BBMR2015,BLV-M2AN,ultimo,CG,CGS17,GMF18,GMF182}. On the other hand, the first work related to spectral problems with mixed formulations is presented in \cite{MZM}  where  a VEM for the mixed formulation of the Laplace eigenvalue problem has been analyzed. For the analysis, the authors lie in the well developed theory of \cite{Boffi} and take advantage of the compact solution operator in order to obtain
convergence of the eigenvalues and eigenfunctions of the Laplace eigenproblem, and therefore, error estimates, using the classic theory of \cite{BO}. Moreover, in this references have analyzed as VEM for virtual spaces BDM-type, where the local spaces are defined for polynomial of degree $k\geq 1$, which have an additional cost compared with the VEM spaces similar to Raviart-Thomas elements (i.e, for $k\geq 0$).
%However, in this paper in particular, the main goal is to derive a priori error estimates. 
%
%On the other hand, the a posteriori error analysis for eigenvalue problems implementing the VEM has been studied in the recent years with 
%important results, as for example \cite{MR2019, MRR2}, just for mention some of them. However, the mixed formulations for spectral problems have not been considered for a posteriori analysis involving the VEM.

We are interested in mixed formulations for eigenvalue problems and as corner stone for more challenging mixed formulations, we begin with the mixed formulation for the Laplace eigenvalue problem in two dimensions. In one hand, we present a rigorous mathematical analysis for the proposed VEM method, which is based the general theory of non-compact operators of \cite{DNR1} in first place,  in order  to prove convergence of our method. The error estimates for the eigenfunctions and eigenvalues will be derived by adapting the results of \cite{DNR2} for the VEM framework and the VEM spaces that we will analyze are of the Raviart-Thomas-type, where the cost of implementation is less than the BDM-type spaces. With this choice of VEM spaces, we will prove that our method is convergent, spurious free and delivers the optimal double order of convergence for the eigenvalues.

On the other hand, it is well known from the literature that some numerical methods that depend on some particular stabilizations may introduce spurious eigenvalues for certain choices of this parameter. Recently for DG methods based in interior penalization,  applied in spectral problems, this phenomena has been studied  in \cite{MR3962898, LM} and also for VEM methods  in \cite{BMRR17, MRR1}. Since our mixed formulation also depends on a stabilization, which is intrinsic in the VEM framework, we will also study from a numerical point of view how this stabilization affects the computation of the spectrum in different polygonal meshes in order to obtain a threshold in which our method works perfectly.

Also, we will discuss our proposed VEM  for a more general Laplace eigenproblem, where the boundary can be splitted in two parts: a Dirichlet and Neumann boundary. This mixed boundary conditions are relevant for our purposes, since the regularity of the solution is clearly affected for this nature of the splitted boundary and hence, the computation of the spectrum may introduce spurious modes, which needs to be controlled by means of the stabilization term of our VEM.

The paper is organized as follows: In section \ref{SEC:STAT} we present the Laplace eigenvalue problem, the mixed formulation for the problem, and we recall important properties of this problem, as the spectral characterization and regularity results. In section \ref{SEC:DISCRETE} we introduce the standard hypothesis for the mesh that the VEM framework requires, the virtual spaces, degrees of freedom and hence, the discrete bilinear forms that are considered for the discrete mixed formulation. 
%Section \ref{SEC:approximation} is dedicated to the complete spectral analysis, where the convergence and error estimates are proved. In section \ref{laplace_mixed} we make a %discussion of the Laplace spectral problem with mixed boundary conditions and the restrictions that are needed for the implementation of our VEM for the computation fo the spectrum. 
Finally, in section \ref{SEC:NUMER}, we report some numerical tests that illustrates the performance of the method and confirms the theoretical results obtained in the previous sections together with a computational analysis of the effects of the stabilization in the computation of the spectrum in a domain with mixed boundary conditions.

% as a new development in VEM for eigenvalue problems, we will design and analyze an a posteriori error estimator for the mixed formulation of the Laplace eigenproblem. 
%----------------------------------------------------------------------

\setcounter{equation}{0}
\section{The spectral problem}
\label{SEC:STAT}

Let $\O\subset\mathbb{R}^2$ be an open bounded domain with
Lipschitz boundary $\Gamma$. The Laplace eigenvalue problem reads as follows:
\begin{problem}
Find $(\l,u)\in\R\times\HuO$, $u\ne0$, such that
\begin{equation}
\label{eq:laplace}
\left\{\begin{array}{l}
-\Delta u=\lambda u\quad\text{in }\O,
\\[0.1cm]
u=0  \,\,\,\quad\quad\text{ on }\G.
%u\cdot\boldsymbol{n}=0  \quad\text{ on }\G_N.
\end{array}\right. 
\end{equation}
\end{problem}

In order to obtain a mixed variational formulation
of \eqref{eq:laplace}, we introduce the additional
unknown $\bsig=\nabla u$. Then, replacing this new
unknown in \eqref{eq:laplace}, multiplying with
suitable test functions, integrating by parts,
and using the boundary condition,  we obtain the
following equivalent mixed weak formulation:

\begin{problem}
\label{P1}
Find $(\l,\bsig, u)\in\R\times\H(\div,\O)\times\LO$,
$(\bsig,u)\neq(\boldsymbol{0},0)$, such that 
\begin{equation*}\label{2}
\begin{split}
\int_{\O}\bsig\cdot\btau+\int_{\O}\div\btau u=&0
\qquad\forall\btau\in\H(\div,\O),\\
\int_{\O}\div\bsig v=&-\lambda \int_{\O}uv\qquad\forall v\in\LO.
\end{split}
\end{equation*}
\end{problem}

We define the spaces $\mathcal{V}:=\H(\div,\O)$ and $\mathcal{Q}=\LO$.
Let us remark that these spaces will be endowed with the usual
norms which we denote by $\|\cdot\|_{\mathcal{V}}$ and
$\|\cdot\|_{\mathcal{Q}}$, respectively, and  the product
space $\mathcal{V}\times Q$ will be endowed with the natural
norm of product spaces which we denote by $\|\cdot\|_{\mathcal{V}\times\mathcal{Q}}$.

With these definitions at hand, we introduce the bilinear
forms $a:\mathcal{V}\times\mathcal{V}\rightarrow\mathbb{R}$
and $b:\mathcal{V}\times \mathcal{Q}\rightarrow\mathbb{R}$,
defined as follows
\begin{align*}
a(\bsig,\btau) :=\int_{\Omega}\bsig\cdot\btau, \quad \bsig,\btau\in\mathcal{V},
\quad
b(\btau,v) :=\int_{\O}v\div\btau, 
\qquad \btau\in\mathcal{V}, \,v\in \mathcal{Q}.
\end{align*}

Then, if $(\cdot,\cdot)_{\mathcal{Q}}$ denotes the usual
$\mathcal{Q}$ inner-product, we rewrite Problem~\ref{P1} as follows:
 \begin{problem}
\label{P2}
Find $(\l,\bsig,u)\in\R\times\mathcal{V}\times \mQ$, $(\bsig,u)\ne(\boldsymbol{0},0)$, such that
\begin{align*}
a(\bsig,\btau)+b(\btau, u)&=0 \qquad\qquad\,\,\forall \btau\in\mathcal{V},\\
b(\bsig,v)&=-\lambda(u,v)_{\mathcal{Q}}\qquad\forall v\in \mathcal{Q}.
\end{align*}
\end{problem}
 We remark that each of the previous bilinear forms are bounded and symmetric.
 
 Let $\mathcal{K}$ be the kernel of bilinear form $b(\cdot,\cdot)$ defined as follows:
 \begin{equation*}
 \mathcal{K}:=\{\btau\in\mathcal{V}\,:\,\, b(\btau, v)=0\quad\forall v\in \mathcal{Q}\}
 =\{\btau\in\mathcal{V}\,:\,\, \div\btau=0\,\,{\textrm in}\,\, \O\}.
 \end{equation*}
 
 Is is well-known that bilinear form $a(\cdot,\cdot)$
 is elliptic in $\mathcal{K}$ and that $b(\cdot,\cdot)$
 satisfies the following inf-sup condition (see \cite{bbf-2013})
 \begin{equation}\label{cont-infsup}
 \displaystyle\sup_{\boldsymbol{0}\neq\btau\in\mathcal{V}}
 \frac{b(\btau,v)}{\|\btau\|_{\mathcal{V}}}\geq\beta\|v\|_{\mathcal{Q}}\qquad\forall v\in \mathcal{Q}, 
 \end{equation}
where $\beta$ is a positive constant.
\begin{remark}
The eigenvalues of Problem~\ref{P2} are positive.
Indeed, taking $\btau=\bsig$ and $v=u$ in Problem~\ref{P2} and
adding the resulting forms, we have obtain
$$\lambda=\frac{a(\bsig,\bsig)}{\|u\|_{\mQ}^2}\geq0.$$
In addition, $\lambda=0$ implies $(\bsig,u)=(\boldsymbol{0},0)$. 
\end{remark}

To analyze Problem~\ref{P2}, we introduce the following linear solution operator $\bT$
\begin{align*}
\nonumber \bT:\;\mathcal{Q} & \longrightarrow\mathcal{Q},
\\
f & \longmapsto\bT f:=\widetilde{u},\label{eq:opT}
\end{align*}
where $(\tilde{\bsig},\tilde{u})\in\mathcal{V}\times\mathcal{Q}$
is the solution of the corresponding source problem:
\begin{equation}
\label{eq:source}
\left\{
\begin{array}{rcc}
a(\widetilde{\bsig},\btau)+b(\btau,\widetilde{u})&=0 \qquad\qquad\,\,\forall \btau\in\mathcal{V},\\
b(\widetilde{\bsig},v)&=-(f,v)_{\mathcal{Q}}\qquad\forall v\in \mathcal{Q},
\end{array}\right.
\end{equation}
which is the variational formulation of the following problem
\begin{equation}
\label{eq:fuerte}
\left\{\begin{array}{l}
\widetilde{\bsig}=\nabla \widetilde{u}\quad\text{in }\O,\\[0.1cm]
\div \widetilde{\bsig}=-f\quad\text{in }\O,
\\[0.1cm]
\widetilde{u}=0  \,\,\,\quad\quad\text{ on }\G.
%u\cdot\boldsymbol{n}=0  \quad\text{ on }\G_N.
\end{array}\right. 
\end{equation}
From the fact that $a(\cdot,\cdot)$ is $\mathcal{K}$-elliptic
and \eqref{cont-infsup}, it is well known that problem~\eqref{eq:source}
admits a unique solution
$(\widetilde{\bsig},\widetilde{u})\in\mathcal{V}\times\mathcal{Q}$
and there exists a positive constant $C$ such that
\begin{equation}
\label{eq:bien}
\|(\widetilde{\bsig},\widetilde{u})\|_{\bV\times\mathcal{Q}}\leq C\|f\|_{Q}.
\end{equation}

As a consequence, we have that $\bT$ is well defined,
self-adjoint with respect to $(\cdot,\cdot)_{\mathcal{Q}}$
and compact. Moreover, if $(\l,(\bsig, u))\in\R\times\mathcal{V}\times\mathcal{Q}$
solves Problem~\ref{P1} if and only if
$(1/\l,u)$ is an eigenpair of $\bT$, i.e, if 
\begin{equation*}
\bT u=\mu u,\quad\text{ with }\mu:=\frac{1}{\l}.
\end{equation*}

According to \cite{agmon}, the regularity for the solution
of \eqref{eq:source} is the following: there exists a constant
$r>1/2$ depending on $\Omega$ such that the solution $\widetilde{u}\in\H^{1+r}(\Omega)$,
where $r$ is at least 1 if $\Omega$ is convex and $r$ is at least
$\pi/\omega-\varepsilon$, for any $\varepsilon>0$ for a non-convex domain,
with $\omega<2\pi$ being the largest reentrant angle of $\Omega$.

Hence we have the following additional regularity result
for the solution of problem~\eqref{eq:source}.

\begin{lemma}
\label{additional_source}
There exist a positive constant $C$ such that
\begin{equation*}
\|\widetilde{\bsig}\|_{r,\O}+\|\widetilde{u}\|_{1+r}\leq C\|f\|_{\mathcal{Q}}.
\end{equation*}
\end{lemma}

On the other hand, since $\bT$ is a self-adjoint compact operator,
we have the following spectral characterization result (see \cite{BO}).
\begin{lemma}
\label{lmm:charact}
The spectrum of $\bT$ satisfies $\sp(\bT)=\{0\}\cup\{\mu_n\,:\,n\in\mathbb{N}\}$,
where $\{\mu_n\}_{n\in\mathbb{N}}$ is a sequence of positive eigenvalues
which converge to zero with the multiplicity of each non-zero eigenvalue being finite.
In addition, the following additional regularity result holds true for eigenfunctions
\begin{equation*}
\label{eq:additional_eigenfunctions}
\|\bsig\|_{\widetilde{r},\O}+\|u\|_{1+\widetilde{r}}\leq C\|u\|_{\mathcal{Q}},
\end{equation*}
with $\widetilde{r}>1/2$ and $C>0$ depending on the eigenvalue.
\end{lemma}

Now we are in position to introduce our approximation scheme.

\setcounter{equation}{0}
\section{The virtual element method}
\label{SEC:DISCRETE}
\subsection{Mesh assumptions and virtual spaces}
We begin this section establishing the framework in which we will operate. The VEM method
needs particular assumptions for the construction of the meshes, which are well established in 
\cite{BBCMMR2013}. Let $\left\{\CT_h\right\}$ be a
family of decompositions of $\O$ into polygons $\E$. Let $h_\E$ denote
the diameter of the element $\E$ and $h:=\max_{\E\in\O}h_\E$.

For the analysis, we make the following assumptions on the meshes as in
\cite{BBMR2015,ultimo}: there exists a positive real number $C_{\CT}$
such that, for every $\E\in\CT_h$ and for every $\CT_h$,
\begin{itemize}
\item $\mathbf{A_1}$: the ratio between the shortest edge and the
diameter of $\E$ is larger than $C_{\CT}$;
\\[-.3cm]
\item $\mathbf{A_2}$: $\E$ is star-shaped with respect to every
point of a ball of radius $C_{\CT}h_\E$.
\end{itemize}

For any subset $S\subseteq\R^2$ and any non-negative integer $k$, we
indicate by $\bbP_k(S)$ the space of polynomials of degree up to $k$
defined on $S$. To keep the notation simpler, we denote by $\bn$ a
generic normal unit vector; in each case, its precise definition will be
clear from the context. 

We consider now a polygon $\E$ and, for any fixed
non-negative integer $k$, we define the following finite dimensional
space (inspired in \cite{ultimo,BBMR2015}):
\begin{equation*}
\label{eq:local_space}
\mathcal{V}_{h}^{\E}:=\Big\{\btau_h\in\H(\div;\E):
\ \left(\btau_h\cdot\bn\right)\in\bbP_k(e)
\ \,\forall e\subset\partial \E,\ \div\btau_h\in\bbP_k(\E),
\ \rot\btau_h=0\text{ in }\E\Big\}.
\end{equation*}

We define the following degrees of freedom for functions $\btau_h$ in
$\mathcal{V}_h^\E$:
\begin{align}
\label{freedom}
\int_{e}\left(\btau_h\cdot\bn\right)q\ds
& \qquad\forall q\in\bbP_k(e),
\quad\forall\text{ edge }e\subset\partial \E,
\\
\label{freedom2}
\int_{\E}\btau_h\cdot\nabla q
& \qquad\forall q\in\bbP_k(\E)/\bbP_{0}(\E).
\end{align}
These degrees of freedom are unisolvent, as is stated in  \cite[Proposition 1]{BMRR17}.

For each decomposition $\CT_h$ of $\O$ into polygons $\E$, we define
\begin{equation}
\mathcal{V}_h:=\left\{\btau_h\in\H(\div,\O):\ \btau_h|_\E\in\bV_h^\E\right\}.
\nonumber
\end{equation}
In agreement with the local choice, we choose the following global
degrees of freedom:
\begin{align*}
\label{globalfreedom}
\int_{e}\left(\btau_h\cdot\bn\right)q\ds
& \qquad\forall q\in\bbP_k(e),
\quad\text{for each internal edge }e\not\subset\G,
\\
\int_{\E}\btau_h\cdot\nabla q
& \qquad\forall q\in\bbP_k(\E)/\bbP_{0}(\E),
\quad\text{for each element }\E\in\CT_h.
\end{align*}
%In order to construct the discrete scheme, we need some preliminary
%definitions. First, we split the bilinear form $b(\cdot,\cdot)$
%introduced in the previous section as follows:
%$$
%b(\widehat{\bsig}_h,\btau_h)
%=\sum_{E\in\CT_h}\left(\int_{E}\div\widehat{\bsig}_h\div\btau_h
%+\int_{E}\widehat{\bsig}_h\cdot\btau_h\right),
%\qquad\widehat{\bsig}_h,\btau_h\in\bV_h.
%$$
Additionally we introduce the following finite dimensional space: 
$$\mathcal{Q}_{h}:=\{v_{h}\in\mathcal{Q}: v_{h}|_{\E}\in \bbP_k(\E),\quad \forall\, \E\in \CT_{h} \}.$$

As is customary in the VEM framework, the bilinear forms $a(\cdot,\cdot)$ and $b(\cdot,\cdot)$ are written elementwise
as follows
\begin{align*}
a(\bsig,\btau)&=\sum_{\E\in\mathcal{T}_h}a^\E(\bsig,\btau)=\sum_{\E\in\mathcal{T}_h}\int_{\E}\bsig\cdot\btau\qquad\bsig,\btau\in\bV_h,\\
b(\bsig, v)&=\sum_{\E\in\mathcal{T}_h}b^\E(\bsig,\btau)=\sum_{\E\in\mathcal{T}_h}\int_{\E}v\div\bsig\qquad\bsig\in\bV,\,\, v\in\mQ.
\end{align*}

Observe that with the degrees of freedom that we are operating, $a^\E(\cdot,\cdot)$ is not explicitly computable, contrary of
$b(\cdot, \cdot)$. For this reason we need to introduce a projection operator to circumvent this drawback.

%The local matrices associated with the first term on the right hand side
%above are easily computable since $\div\widehat{\bsig}_h$ and $\div\btau_h$ are
%polynomials in each element. We explicitly point out that, as can be
%seem from \eqref{freedom}--\eqref{freedom2}, the divergence of any
%vector $\btau_h\in\btau_h$ can be easily computed from knowledge of the
%degrees of freedom of $\btau_h$. Instead, for the local matrices
%associated with the second term on the right hand side above, we must
%take into account that, due to the implicit space definition, it is not
%possible to compute exactly the integrals. Because of this, we will use
%an approximation of them. The final output will be a local matrix on
%each element $E$ whose associated bilinear form is exact whenever one of
%the two entries is a gradient of a polynomial of degree $k+1$. This will
%allow us to retain the optimal approximation properties of the space
%$\btau_h$.
First, we define for each polygon $\E$ the space
\begin{equation*}
\label{Ve}
\widehat{\bV}_h^\E
:=\nabla(\bbP_{k+1}(\E))
\subset\bV_{h}^{\E}.
\end{equation*}
Then, we define the $[\L^2(\E)]^2$-orthogonal projector
$\bpi_h^\E:\;[\L^2(\E)]^2\longrightarrow\widehat{\bV}_h^\E$ by
\begin{equation*}
\label{numero}
\int_{\E}\bpi_h^\E\btau\cdot\widehat{\bu}_h
=\int_{\E}\btau\cdot\widehat{\bu}_h
\qquad\forall\widehat{\bu}_h\in\widehat{\bV}_h^\E.
\end{equation*}
%We point out that $\bpi_h^E\btau_h$ is explicitly computable for every
%$\btau_h\in\bV_h^E$ using only its degrees of freedom
%\eqref{freedom}--\eqref{freedom2}. In fact, it is easy to check that for
%all $\btau_h\in\bV_h^E$ and for all $ q\in\bbP_{k+1}(E)$,
%\begin{equation*}
%\int_{E}\bpi_h^E\btau_h\cdot\nabla q
%=\int_{E}\btau_h\cdot\nabla q
%=-\int_{E}\div\btau_h\,q
%+\int_{\partial E}\left(\btau_h\cdot\bn\right)q\ds.
%\end{equation*}

%\begin{remark}
%In particular, for $k=0$, for all $\btau_h\in\bV_h^E$ and for all
%$q\in\bbP_{1}(E)$, we have that
%$$
%\int_{E}\bpi_h^E\btau_h\cdot\nabla q
%=-\left(\dfrac{1}{\left|E\right|}
%\sum_{e\subset\partial E}\int_{e}\btau_h\cdot\bn\ds\right)
%\left(\int_{E}q\right)
%+\sum_{e\subset\partial E}\int_{e}\left(\btau_h\cdot\bn\right)q\ds.
%$$
%\end{remark}

Let $S^\E(\cdot,\cdot)$ be any symmetric positive
definite (and computable) bilinear form to be chosen as to satisfy
\begin{equation}
\label{20}
c_0\,a^\E(\btau_h,\btau_h)\leq S^{\E}(\btau_h,\btau_h)
\leq c_1\,a^\E(\btau_h,\btau_h)
\qquad\forall\,\bV_h\in\btau_h^\E,
\end{equation}
for some positive constants $c_0$ and $c_1$ depending only on the 
constant $C_{\CT}$ from mesh assumptions $\mathbf{A_1}$ and
$\mathbf{A_2}$. Then, we define on each element $\E$ the bilinear form
\begin{equation*}
\label{21}
a_h^{\E}(\bsig_h,\btau_h)
:=\int_{\E}\bpi_h^\E\bsig_h\cdot\bpi_h^\E\btau_h
+S^{\E}\big(\bsig_h-\bpi_h^\E\bsig_h,\btau_h-\bpi_h^\E\btau_h\big),
\qquad\bsig_h,\btau_h\in\bV_h^\E, 
\end{equation*}
and, in a natural way,
$$
a_h(\bsig_h,\btau_h)
:=\sum_{\E\in\CT_h}a_h^{\E}(\bsig_h,\btau_h),
\qquad\bsig_h,\btau_h\in\bV_h.
$$
The following two properties of the bilinear form $a_h^\E(\cdot,\cdot)$
are easily derived by repeating in our case the arguments from
\cite[Proposition~4.1]{ultimo}.
\begin{itemize}
\item \textit{Consistency}: 
\begin{equation*}
\label{consistencia1}
a_h^{\E}(\bu_h,\btau_h)
=\int_{\E}\bu_h\cdot\btau_h
\qquad\forall\bu_h\in\widehat{\bV}_h^\E,
\quad\forall\btau_h\in\bV_h^\E,\quad\forall\, \E\in\CT_h.
\end{equation*}
\item \textit{Stability}: There exist two positive constants $\alpha_*$
and $\alpha^*$, independent of $\E$, such that:
\begin{equation}
\label{consistencia2}
\alpha_*\int_{\E}\btau_h\cdot\btau_h
\leq a_h^{\E}(\btau_h,\btau_h)
\leq\alpha^*\int_{\E}\btau_h\cdot\btau_h
\qquad\forall\btau_h\in\bV_h^\E,\quad\forall E\in\CT_h. 
\end{equation}
\end{itemize}

Now we are in a position to introduce the virtual element discretization of
Problem~\ref{P1}.

\subsection{The discrete eigenvalue problem}
With the VEM spaces and degrees of freedom defined above, we introduce the discretization of Problem~\ref{P2}
as follows
%\begin{problem}
%\label{P1_disc}
%Find $(\l_h,\bsig_h)\in\R\times\bV_h$, $\bsig_h\neq 0$, such that 
%\begin{equation*}
%\label{2}
%\int_{\G}(\bsig_h\cdot\bn)(\btau_h\cdot\bn)=\l\left(\int_{\O}\bsig_h\cdot\btau_h+\int_{\O}\div\bsig_h\div\btau_h \right)
%\qquad\forall\btau_h\in\bV_h.
%\end{equation*}
%\end{problem}
%%where
%%$$
%%b_h(\bw_h,\btau_h)
%%:=\int_{\O}\div\bw_h\div\btau_h
%%+c_h(\bw_h,\btau_h)
%%\qquad\forall\bw_h,\btau_h\in\btau_h.
%%$$
%%\end{problem}
%
%We use again a shift argument to rewrite this discrete eigenvalue
%problem in the following convenient equivalent form.
\begin{problem}
\label{P2_disc}
Find $(\l_h, \bsig_h, u_h)\in\R\times\mathcal{V}_h\times \mQ_h$, $(\bsig_h,u_h)\ne(\boldsymbol{0},0)$, such that
\begin{align*}
a_h(\bsig_h,\btau_h)+b(\btau_h, u_h)&=0 \qquad\qquad\,\,\,\,\,\,\,\,\,\,\forall \btau_h\in\mathcal{V}_h,\\
-b(\bsig_h,v_h)&=\lambda_h(u_h,v_h)_{\mQ}\qquad\forall v_h\in \mathcal{Q}_h.
\end{align*}
\end{problem}
%where
%%with
%%$$
%%a_h(\bw_h,\btau_h)
%%:=\int_{\G}(\bsig_{h}\cdot\bn)(\btau_{h}\cdot\bn)
%%+b_h(\bw_h,\btau_h)
%%\qquad\forall\bw_h,\btau_h\in\btau_h.
%%$$
%\begin{align*}
%b_h(\bsig_h,\btau_h)=\sum_{E\in\mathcal{T}_h}b_h^E(\bsig_h,\btau_h)
%& :=\sum_{E\in\mathcal{T}_h}\left(\int_{E}\bsig_h\cdot\btau_h+\int_{E}\div\bsig_h\div\btau_h\right) 
%\qquad \bsig_h,\btau_h\in\bV_h,\\
%a_h(\bsig_h,\btau_h)=\sum_{E\in\mathcal{T}_h}a_h^E(\bsig_h,\btau_h)
%& :=\sum_{E\in\mathcal{T}_h}b_h^E(\bsig_h,\btau_h)+\int_{\partial E\subset\Gamma}(\bsig_h\cdot\bn)(\btau_h\cdot\bn)
%\qquad \bsig_h,\btau_h\in\bV_h,
%\end{align*}
%\end{problem}
%Analogously as the continuous case, we define the bilinear form  $c_h:\bV_h\times\bV_h\rightarrow\mathbb{R}$ as follows
%\begin{equation*}
%c_h(\bsig_h,\btau_h)=\sum_{E\in\mathcal{T}_h}c_h^E(\bsig_h,\btau_h)
% :=\sum_{E\in\mathcal{T}_h}\int_{E}\bsig_h\cdot\btau_h
%\qquad \bsig_h,\btau_h\in\bV_h.
%\end{equation*}
  Let $\mathcal{K}_h$ be the discrete kernel of bilinear form $b(\cdot, \cdot)$ defined as follows:
 \begin{equation*}
 \mathcal{K}_h:=\{\btau_h\in\mathcal{V}_h\,:\,\, b(\btau_h, v_h)=0\quad\forall\, v_h\in \mathcal{Q}_h\}.
 \end{equation*}

We observe that by virtue of \eqref{consistencia2}, the bilinear form
$a_h(\cdot,\cdot)$ is bounded. Moreover, as is shown in the following
lemma, it is also uniformly elliptic.

\begin{lemma}
\label{ha-elipt-disc}
There exists a constant $\beta>0$, independent of $h$, such that
$$
a_h(\btau_h,\btau_h)
\ge\alpha\left\|\btau_h\right\|_{\bV}^2
\qquad\forall\,\btau_h\in\mathcal{K}_h.
$$
\end{lemma}

\begin{proof}
Thanks to \eqref{consistencia2}, the above inequality holds with
$\alpha:=\min\left\{\alpha_{*},1\right\}$.
\end{proof}

Also, the following discrete inf-sup condition holds.
\begin{lemma}
\label{lmm:disc_infsup}
There exists $\widehat{\beta}>0$, independent of $h$, such that
 \begin{equation*}
 \displaystyle\sup_{\boldsymbol{0}\neq\btau_h\in\mathcal{V}_h}\frac{b(\btau_h,v_h)}{\|\btau_h\|_{\mathcal{V}_h}}\geq\widehat{\beta}\|v_h\|_{\mathcal{Q}_h}\qquad\forall v_h\in \mathcal{Q}_h.
 \end{equation*}
 \end{lemma}
 \begin{proof}
The proof is straightforward by adapting the arguments of \cite[Lemma 5.3]{CG}.
 \end{proof}

The next step is to introduce the discrete version of  the operator
$\bT$:
\begin{align*}
\nonumber \bT_h:\;\mQ_h & \longrightarrow\mQ_h,
\\
f_h & \longmapsto\bT_h f_h:=\widetilde{u}_h,\label{eq:opTh}
\end{align*}
where $(\widetilde{\bsig}_h, \widetilde{u}_h)\in\bV_h\times\mQ_h$ is the solution of the corresponding discrete
source problem:
\begin{equation}
\label{eq:source_disc}
\left\{
\begin{array}{rcc}
a_h(\widetilde{\bsig}_h,\btau_h)+b(\btau_h, \widetilde{u}_h)&=0 \qquad\qquad\,\,\forall \btau_h\in\mathcal{V}_h,\\
-b(\widetilde{\bsig}_h,v_h)&=(f_h,v_h)_{\mQ}\qquad\forall v_h\in \mathcal{Q}_h.
\end{array}\right.
\end{equation}
%We deduce from Lemma~\ref{ha-elipt-disc}, \eqref{consistencia2} and the
%Lax-Milgram Theorem, that the linear operator $\bT_h$ is well defined
%and bounded uniformly with respect to $h$.

%We observe that by virtue of \eqref{24} and the trace theorem, the
%bilinear form $\ha_{h}(\cdot,\cdot)$ is bounded. Moreover, as shown in
%the following lemma, it is also uniformly elliptic.
%
%\begin{lemma}
%\label{ha-elipt-disc}
%There exists a constant $\beta>0$, independent of $h$, such that
%$$
%\ha_{h}(v_h,v_h)
%\ge\beta\left\|v_h\right\|_{1,\O}^2
%\qquad\forall v_h\in\bV_h.
%$$
%\end{lemma}
%
%\begin{proof}
%Thanks to \eqref{24} and Lemma~\ref{ha-elipt}, it is easy to check that
%the above inequality holds with
%$\beta:=\alpha\min\left\{\alpha_{*},1\right\}$.
%\end{proof}

%-----------------------------------------------------------------------
In what follows, we state some auxiliar results about the approximation properties
of this interpolant (see \cite{BMRR17}). The first
one concerns approximation properties of $\div\bv_I$ and follows from a
commuting diagram property for this interpolant, which involves the
$\LO$-orthogonal projection 
\begin{equation*}
\label{eq:L2project}
P_k:\;\LO\longrightarrow
\left\{q\in\LO: q|_\E\in\bbP_k(\E)\quad\forall\, \E\in\CT_h\right\}.
\end{equation*}
For $P_k$ we have the following approximation estimate (see \cite{BBMR2015}): if $0\leq s \leq k+1$,  it holds
\begin{equation}
\label{eq:salim}
\|v-P_k(v)\|_{0,\O}\leq Ch^{s}\|v\|_{s,\O}\qquad\forall s\in\H^s(\O)\cap\mQ.
\end{equation}

\begin{lemma}
\label{lemmainter}
Let $\btau\in\bV$ be such that $\btau\in[\HtO]^2$ with $t>1/2 $. Let
$\btau_I\in\bV_h$ be its interpolant defined by \eqref{freedom}--\eqref{freedom2}.
Then,
$$
\div\btau_I=P_k(\div\btau)\quad\text{ in }\O.
$$
Consequently, for all $\E\in\CT_h$,
$\left\|\div\btau_I\right\|_{0,\E}\leq\left\|\div\btau\right\|_{0,\E}$ and,
if $\div\btau|_{\E}\in\HrE$ with $r\geq 0$, then
\begin{equation*}
\left\|\div\btau-\div\btau_I\right\|_{0,\E}
\leq Ch_\E^{\min\{r,k+1\}}\left|\div\btau\right|_{r,\E}. 
\end{equation*}
\end{lemma}
\begin{proof}
See \cite[Appendix]{BMRR17}.
\end{proof}
The second result concerns the $\LO$ approximation property of $\btau_I$.

\begin{lemma}
\label{lemmainterV_I}
Let $\btau\in\bV$ be such that $\btau\in[\HtO]^2$ with $t>1/2$. Let
$\btau_I\in\bV_h$ be its interpolant defined by \eqref{freedom}--\eqref{freedom2}.
Let $\E\in\CT_h$. If $1\leq t\leq k+1$, then
\begin{equation*}
\left\|\btau-\btau_I\right\|_{0,\E}
\leq Ch_\E^{t}\left|\btau\right|_{t,\E},
\end{equation*}
whereas, if $1/2 <t\leq1$, then
\begin{equation*}
\left\|\btau-\btau_I\right\|_{0,\E}
\leq C\left(h_\E^{t}\left|\btau\right|_{t,\E}
+h_\E\left\|\div\btau\right\|_{0,\E}\right).
\end{equation*}
\end{lemma}
\begin{proof}
See \cite[Appendix]{BMRR17}.
\end{proof}
%Moreover, we recall the following result which is derived by interpolation between Sobolev spaces (see, for instance, \cite{BL1976}) from the analogous result for integer values of s. In its turn, the result for integer values is stated in \cite[Proposition 4.2]{BBCMMR2013} and follows from the classical Scott-Dupont theory (see \cite{BS-2008}).
The end this section by recalling the following technical result.
\begin{lemma}
\label{Acustic}
There exists a constant $C>0$ such that, for every $p\in\HutO$ with
$1/2<t\le k+1$, there holds 
$$
\left\|\nabla p-\bpi_h (\nabla p)\right\|_{0,\O}
\leq Ch^{t}\left\|\nabla p\right\|_{t,\O},
$$
where  $\left(\bpi_h\bv\right)|_\E
:=\bpi_h^\E(\bv|_\E)
\quad\text{ for all }\E\in\CT_h.
$
\end{lemma}
\begin{proof}
See \cite[Lemma 8]{BMRR17}.
\end{proof}

%\begin{proposition}
%There exists a constant $C>0$, such that for every $\btau\in [\H^{s}(\E)]^{2}$, there exists $\btau_{\pi}\in \mathbb{P}_{k}(E)$, $k\geq 0$, such that
%$$\|\btau-\btau_{\pi}\|_{l,\E}\leq Ch_{\E}^{s-l}|\btau|_{s,\E},\quad 0\leq s\leq k+1, l=0,...,[s]$$
%with $[s]$ denoting the largest integer equal to or smaller than $s\in\mathbb{R}$.
%\end{proposition}
\setcounter{equation}{0}
\section{Spectral approximation}
\label{SEC:approximation}
In what follows, we will prove that convergence properties for the numerical method proposed in Section \ref{SEC:DISCRETE}. 
We begin this section by recalling some definitions of spectral theory. 

Let $\mathcal{X}$ be a generic Hilbert space and let $\bS$ be a linear bounded operator defined by $\boldsymbol{S}:\mathcal{X}\rightarrow\mathcal{X}$. If $\boldsymbol{I}$ represents the identity operator, the spectrum of $\boldsymbol{S}$ is defined by $\sp(\boldsymbol{S}):=\{z\in\mathbb{C}:\,\,(z\boldsymbol{I}-\boldsymbol{S})\,\,\text{is not invertible} \}$ and the resolvent is its complement $\rho(\bS):=\mathbb{C}\setminus\sp(\boldsymbol{S})$. For any $z\in\rho(\boldsymbol{S})$, we define the resolvent operator of $\boldsymbol{S}$ corresponding to $z$ by $R_z(\boldsymbol{S}):=(z\boldsymbol{I}-\boldsymbol{S})^{-1}:\mathcal{X}\rightarrow\mathcal{X}$.

Despite to the fact that $\bT$ is compact, since the discrete solution operator is defined from 
$\mQ_h$ onto itself,  the non-compact theory of \cite{DNR1} is suitable for this setting.

We introduce the following definition
\begin{equation*}
\displaystyle \|\bT\|_h:=\sup_{0\neq f_h\in\mQ_h}\frac{\|\bT f_h\|_{\mQ}}{\|f_h\|_{\mQ}}.
\end{equation*}

No we recall properties P1 and P2 of \cite{DNR1}.

\begin{itemize}
\item P1: $\|\bT-\bT_h\|_h\rightarrow 0$ as $h\rightarrow 0$;
\item P2: $\forall\btau\in\mQ$, $\displaystyle \inf_{\btau_h\in\mQ_h}\|\btau-\btau_h\|_{\mQ}\rightarrow 0$ as $h\rightarrow 0$.
\end{itemize}

Our task consists into prove properties P1 and P2 in order to ensure the spectral convergence. We observe that P2 is an immediate consequence from the fact that the smooth functions are dense in $\mQ$. Hence, only remains to prove property P1.

\begin{lemma}
\label{lemcotste}
There exists $C>0$ such that, for all $f_{h}\in\mQ_h$, if $\widetilde{u}=Tf_{h}$ and
$\widetilde{u}_h=T_hf_{h}$, then
$$
\left\|\left(T-T_h\right)f_{h}\right\|_{h}
=\left\|\widetilde{u}-\widetilde{u}_h\right\|_{h}\le C h^{r}.
$$
%for all $u_I\in\Vh$ and for all $u_{\pi}\in\LO$ such that
%$u_{\pi}|_{\E}\in\P_k(\E)$ $\forall\E\in\CT_h$.
\end{lemma}

\begin{proof}
Let $f\in\mQ_h$ sucht that $\widetilde{u}=Tf_{h}$, $\widetilde{u}_h=T_hf_{h}$  and $\widetilde{\bsig}_{I}\in \Vh$. From triangular inequality we have,
\begin{equation*}
\|\widetilde{\bsig}-\widetilde{\bsig}_{h}\|_{0,\O}\leq \|\widetilde{\bsig}-\widetilde{\bsig}_{I}\|_{0,\O}+\|\widetilde{\bsig}_{I}-\widetilde{\bsig}_{h}\|_{0,\O}.
\end{equation*}
 We set $\btau_h:=\widetilde{\bsig}_I-\widetilde{\bsig}_h$, thanks to
Lemma~\ref{lemmainter},  equations \eqref{eq:source} and \eqref{eq:source_disc}, we have $\div \widetilde{\bsig}_{I}=P_k(\div \widetilde{\bsig})=f_{h}=\div \widetilde{\bsig}_{h}$, then $\div\btau_{h}=0$. Hence $\btau_{h}\in \mathcal{K}_h\subset \mathcal{K}$. Therefore, we have
\begin{align*}
\alpha\left\|\btau_h\right\|^2_{0,\O}&=\alpha\left\|\btau_h\right\|^2_{\mathcal{V}}
 \leq a_h(\widetilde{\bsig}_I,\btau_h)-a_h(\widetilde{\bsig}_h,\btau_h)= a_h(\widetilde{\bsig}_I,\btau_h)+b(\btau_h,u_{h})\\
& =\sum_{\E\in\CT_h}\left[a_h^\E(\widetilde{\bsig}_I-\bpi_h^\E\widetilde{\bsig},\btau_h)+a^{\E}(\bpi_h^\E\widetilde{\bsig}-\widetilde{\bsig},\btau_h)\right] +a(\widetilde{\bsig},\btau_{h})
\\
& =\sum_{\E\in\CT_h}\left[a_h^\E(\widetilde{\bsig}_I-\bpi_h^\E\widetilde{\bsig},\btau_h)+a^{\E}(\bpi_h^\E\widetilde{\bsig}-\widetilde{\bsig},\btau_h)\right] 
\\
&\leq C\sum_{\E\in\CT_h}\left(\|\widetilde{\bsig}-\widetilde{\bsig}_{I}\|_{0,\E}+\|\widetilde{\bsig}-\bpi_h^\E\widetilde{\bsig}\|_{0,\E}\right)\|\btau_{h}\|_{0,\O}. 
\end{align*}
Therefore, we obtain
\begin{equation}
\label{eq:sigma}
\|\widetilde{\bsig}-\widetilde{\bsig}_{h}\|_{0,\O}\leq C\left(\sum_{\E\in\CT_h}\left(\|\widetilde{\bsig}-\widetilde{\bsig}_{I}\|_{0,\E}+\|\widetilde{\bsig}-\bpi_h^\E\widetilde{\bsig}\|_{0,\E}\right)\right).
\end{equation}
The next step is  to control $\|\left(T-T_h\right)f_{h}\|_{0,\O}$.  Again, using triangle inequality we obtain
\begin{align}
\label{eq:estimateP1}
\|\left(T-T_h\right)f_{h}\|_{0,\O}=\|\widetilde{u}-\widetilde{u}_{h}\|_{0,\O}\leq \|\widetilde{u}-P_{k}(\widetilde{u})\|_{0,\O}+\|P_{k}(\widetilde{u})-\widetilde{u}_{h}\|_{0,\O}.
\end{align}  
Now, adapting the arguments of \cite[Lemma 5.3]{CG}, we deduce that there exists $\widehat{\bsig}_{h}\in \mathcal{V}_{h}$ such that
\begin{align}
\label{eq:GG}
\div\widehat{\bsig}_{h}=P_{k}(\widetilde{u})-\widetilde{u}_{h}\qquad\text{and}\qquad \|\widehat{\bsig}_{h}\|_{\mathcal{V}}\leq c\|P_{k}(\widetilde{u})-\widetilde{u}_{h}\|_{0,\O}.
\end{align}
Hence, 
\begin{align*}
\|P_{k}(\widetilde{u})-\widetilde{u}_{h}\|_{0,\O}^{2}&=\int_{\O}\left(P_{k}(\widetilde{u})-\widetilde{u}_{h}\right)\div\widehat{\bsig}_{h}=\int_{\O}\left(\widetilde{u}-\widetilde{u}_{h}\right)\div\widehat{\bsig}_{h}\\
&=b(\widehat{\bsig}_{h},\widetilde{u})-b(\widehat{\bsig}_{h},\widetilde{u}_{h})=a_{h}(\widetilde{\bsig}_{h},\widehat{\bsig}_{h})-a(\widetilde{\bsig},\widehat{\bsig}_{h})\\
&= \sum_{\E\in\CT_h}\left[a_h^\E(\widetilde{\bsig}_h-\bpi_h^\E\widetilde{\bsig},\widehat{\bsig}_{h})-a^{\E}(\widetilde{\bsig}-\bpi_h^\E\widetilde{\bsig},\widehat{\bsig}_{h})\right] \\
&\leq C \sum_{\E\in\CT_h}\left(\|\widetilde{\bsig}_h-\widetilde{\bsig}\|_{0,\E}+\|\widetilde{\bsig}-\bpi_h^\E\widetilde{\bsig}\|_{0,\E}\right)\|\widehat{\bsig}_{h}\|_{\mathcal{V}} 
\end{align*}
It follows of the above estimate, \eqref{eq:sigma}, \eqref{eq:GG} and \eqref{eq:estimateP1}
%\begin{align*}
%\|\left(T-T_h\right)f_{h}\|_{0,\O}&\leq C\left(\|\widetilde{u}-P_{k}(\widetilde{u})\|_{0,\O}+\sum_{\E\in\CT_h}\left(\|\widetilde{\bsig}-\widetilde{\bsig}_{I}\|_{0,\E}+\|\widetilde{\bsig}-\bpi_h^\E\widetilde{\bsig}\|_{0,\E}\right)\right)\\
%&\leq C\left(\sum_{\E\in\CT_h}\left(h_{\E}^{1+r}|\widetilde{u}|_{1+r,\E}+h_{\E}^{r}|\widetilde{\bsig}|_{r,\E}+h_{\E}\|\div\widetilde{\bsig}\|_{0,\E}\right)\right)\\
%&\leq Ch^{r}\left(\|\widetilde{u}\|_{1+r,\O}+\|\widetilde{\bsig}\|_{r,\O}+\|\div\widetilde{\bsig}\|_{0,\O}\right)\leq Ch^{r}\|f_{h}\|_{0,\O},
%\end{align*}
\begin{align}
\label{eq: errorT}
\|\left(T-T_h\right)f_{h}\|_{0,\O}&\leq C\left(\|\widetilde{u}-P_{k}(\widetilde{u})\|_{0,\O}+\sum_{\E\in\CT_h}\left(\|\widetilde{\bsig}-\widetilde{\bsig}_{I}\|_{0,\E}+\|\widetilde{\bsig}-\bpi_h^\E\widetilde{\bsig}\|_{0,\E}\right)\right)
%&\leq C\left(\sum_{\E\in\CT_h}\left(h_{\E}^{1+r}|\widetilde{u}|_{1+r,\E}+h_{\E}^{r}|\widetilde{\bsig}|_{r,\E}+h_{\E}\|\div\widetilde{\bsig}\|_{0,\E}\right)\right)\\
%&\leq Ch^{r}\left(\|\widetilde{u}\|_{1+r,\O}+\|\widetilde{\bsig}\|_{r,\O}+\|\div\widetilde{\bsig}\|_{0,\O}\right)\leq Ch^{r}\|f_{h}\|_{0,\O},
\end{align}
Now, we need to estimate the three terms on the right-hand side above.
For the first term, invoking \eqref{eq:salim} we obtain
\begin{equation}
\label{eq:firstT}
\|\widetilde{u}-P_{k}(\widetilde{u})\|_{0,\O}\leq C h_{\E}^{1+r}\|\widetilde{u}\|_{1+r,\O}.
\end{equation}
For the second term, we using Lemma \ref{lemmainter} and \ref{lemmainterV_I}, we have
\begin{equation*}
\label{eq:secondT}
\sum_{\E\in\CT_h}\|\widetilde{\bsig}-\widetilde{\bsig}_{I}\|_{0,\E}\leq C\left(\sum_{\E\in\CT_h}\left(h_{\E}^{r}|\widetilde{\bsig}|_{r,\E}+h_{\E}\|\div\widetilde{\bsig}\|_{0,\E}\right)\right)\leq 
\end{equation*}
Finally for the third term, using $\widetilde{\bsig}=\nabla(\widetilde{u})$ (see \eqref{eq:fuerte}) and Lemma  \ref{Acustic}, we obtain 
\begin{align}
\label{eq:thirdT}
\sum_{\E\in\CT_h}\|\widetilde{\bsig}-\bpi_h^\E\widetilde{\bsig}\|_{0,\E}=\sum_{\E\in\CT_h}\|\nabla (\widetilde{u})-\bpi_h^\E\nabla(\widetilde{u})\|_{0,\E}\leq C\sum_{\E\in\CT_h}h_{\E}^{r}|\widetilde{\bsig}|_{r,\E}
%&\leq C\left(\sum_{\E\in\CT_h}\left(h_{\E}^{1+r}|\widetilde{u}|_{1+r,\E}+h_{\E}^{r}|\widetilde{\bsig}|_{r,\E}+h_{\E}\|\div\widetilde{\bsig}\|_{0,\E}\right)\right)\\
%&\leq Ch^{r}\left(\|\widetilde{u}\|_{1+r,\O}+\|\widetilde{\bsig}\|_{r,\O}+\|\div\widetilde{\bsig}\|_{0,\O}\right)\leq Ch^{r}\|f_{h}\|_{0,\O},
\end{align}

Substituting \eqref{eq:firstT}--\eqref{eq:thirdT} in \eqref{eq: errorT} and using \eqref{eq:bien}, we have 
\begin{align*}
\|\left(T-T_h\right)f_{h}\|_{0,\O}\leq Ch^{r}\left(\|\widetilde{u}\|_{1+r,\O}+\|\widetilde{\bsig}\|_{r,\O}+\|\div\widetilde{\bsig}\|_{0,\O}\right)\leq Ch^{r}\|f_{h}\|_{0,\O}.
\end{align*}
Hence we conclude the proof.
\end{proof}

 As a consequence of P1, we have the following results (see \cite[Lemma 1 and Theorem 1]{DNR1}).
 
 The first of these results establishes that the discrete resolvent is bounded.
 \begin{lemma}
 \label{lmm:resolvent1}
  Assume that P1 hold. Let $F\subset\rho(\bT)$ be closed. Then, there exist
 positive constants $C$ and $h_0$, independent of $h$, such that for $h<h_0$
 \begin{equation*}
 \displaystyle\sup_{v_h\in\mQ_h}\|R_z(\bT_h)v_h\|_{\mQ}\leq C\|v_h\|_{\mQ}\qquad\forall z\in F.
 \end{equation*}
 \end{lemma}
 The following results establishes that the numerical method does not introduce spurious eigenvalues.
 \begin{theorem}
 \label{thm:spurious_free}
 Let $U\subset\mathbb{C}$ be an open set containing $\sp(\bT)$. Then, there exists $h_0>0$ such that $\sp(\bT_h)\subset U$ for all $h<h_0$.
 \end{theorem}

As a consequence of the previous results is that the proposed numerical
method does not introduces spurious eigenvalues. Moreover, according to \cite[Section 2]{DNR1} we have the spectral convergence of $\bT_h$ to $\bT$ as $h$ goes to zero. In fact, if $\mu\in (0,1)$ is an isolated eigenvalue of $\bT$ with multiplicity $m$ and $\mathcal{C}$ is an open circle on the complex plane centered at $\mu$ with boundary $\gamma$, we have that $\mu$ is the only eigenvalue of $\bT$ lying in $\mathcal{C}$ and $\gamma\cap\sp(\bT)=\emptyset$. 
Also, invoking \cite[Section 2]{DNR1}, we deduce that for $h$ small enough there exist $m$ eigenvalues $\mu_h^{1},\ldots, \mu_h^{m}$ of $\bT_h$ (according to their respective multiplicities) that lie in $\mathcal{C}$ and hence, the eigenvalues $\mu_h^{i}$, $i=1,\ldots, m$ converge to $\mu$ as $h$ goes to zero. 
\subsection{Error estimates}
\label{SEC:BUCKL:SPEC}

As a direct consequence of Lemma~\ref{lemcotste}, standard results about
spectral approximation (see \cite{K}, for instance) show that isolated
parts of $\sp(T)$ are approximated by isolated parts of $\sp(T_h)$. More
precisely, let $\mu\in(0,1)$ be an isolated eigenvalue of $T$ with
multiplicity $m$ and let $\CE$ be its associated eigenspace. Then, there
exist $m$ eigenvalues $\mu^{(1)}_h,\dots,\mu^{(m)}_h$ of $T_h$ (repeated
according to their respective multiplicities) which converge to $\mu$.
Let $\CE_h$ be the direct sum of their corresponding associated
eigenspaces.

We recall the definition of the \textit{gap} $\hdel$ between two closed
subspaces $\CM$ and $\CN$ of $\L^2(\O)$:
$$
\hdel(\CM,\CN)
:=\max\left\{\delta(\CM,\CN),\delta(\CN,\CM)\right\},
\qquad\text{where}\quad
\delta(\CM,\CN)
:=\sup_{x\in\CM:\ \left\|x\right\|_{\mQ}=1}
\left(\inf_{y\in\CN}\left\|x-y\right\|_{\mQ}\right).
$$

We define
\begin{equation*}
B_h:=\bT_h P_k:\mQ\rightarrow \mQ,
\end{equation*}
such that $B_h$ and $\bT_h$ have the same non-zero eigenvalues and corresponding eigenfunctions.

Let $E:\mQ\rightarrow\mQ$ be the spectral projector of $\bT$ corresponding to the isolated 
eigenvalue $\mu$, namely
\begin{equation*}
\displaystyle E:=\frac{1}{2\pi i}\int_{\gamma} R_{z}(\bT)dz.
\end{equation*}
On the other, we define $F_h:\mQ\rightarrow\mQ$ as the spectral projector of $\bT_h$ corresponding to the isolated 
eigenvalue $\mu_h$, namely
\begin{equation*}
\displaystyle F_h:=\frac{1}{2\pi i}\int_{\gamma} R_{z}(B_h)dz.
\end{equation*}

From \cite[Lemma 1]{DNR1} we have the following result.
\begin{lemma}
\label{lmm:Fbound}
There exist strictly positive constants $h_0$ and $C$ such that
\begin{equation*}
\|R_z(B_h)\|\leq C\qquad\forall h<h_0,\quad\forall z\in\gamma.
\end{equation*}
\end{lemma}

The following result will be used to prove the convergence between the continuous and discrete eigenspaces.
\begin{lemma}
\label{lemma:F_T}
There exist positive constants $C$ and $h_0$ such that, for all $h<h_0$, the following estimates hold
\begin{equation*}
\|(E-F_h)|_{E(\mQ)}\| \leq C\|(\bT-B_h)|_{E(\mQ)}\| \leq Ch^{\min\{\widetilde{r},k\}},
\end{equation*}
where $\widetilde{r}>1/2$ is such that $\CE\subset [H^{\widetilde{r}}(\O)]$ (cf. Lemma \ref{lmm:charact}).
\end{lemma}
\begin{proof}
The first estimate is straightforward from \cite[Lemma 3]{DNR2} and Lemma \ref{lmm:resolvent1}. For the second estimate we procede as follows: Let $u\in E(\mQ)$. Then
\begin{align*}
\|(T-B_h)u\|_{\mQ}&\leq \|(T-TP_k)u\|_{\mQ}+\|(TP_k-B_h)u\|_{\mQ}\\
                            &=\|T(I-P_k)u\|_{\mQ}+\|(T-T_h)P_ku\|_{\mQ}\\
                            &\leq \|T\|\|(I-P_k)u\|_{\mQ}+\|T-T_h\|\|P_ku\|_\mQ\leq Ch^{\min\{\widetilde{r},k\}}\|u\|_{\mQ},
\end{align*}
where we have used triangular inequality, Lemma \ref{lemcotste}, the fact that $T$ is bounded and that $P_k$ is the $\L^2$-projection. This concludes the proof.
\end{proof}

The following error estimates for the approximation of eigenvalues and
eigenfunctions hold true.

\begin{theorem}
\label{gap}
There exists a strictly positive constant $C$ such that
\begin{align*}
\hdel(F_h(\mQ),E(\mQ)) 
& \le C\xi_h,
\\
\left|\mu-\mu_h^{(i)}\right|
& \le C\xi_h,\qquad i=1,\dots,m,
\end{align*}
where
$$
\xi_h
:=\sup_{f\in E(\mathcal{Q}):\ \left\|f\right\|_{\mQ}=1}
\left\|(T-T_h)f\right\|_{\mQ}.
$$
\end{theorem}

\begin{proof}
The proof runs identically as in \cite[Theorem 1]{DNR2}.
%As a consequence of Lemma~\ref{lemcotste}, $T_h$ converges in norm to $T$
%as $h$ goes to zero. Then, the proof follows as a direct consequence of
%Theorems~7.1 and 7.3 from \cite{BO}.
\end{proof}

The next
step is to show an optimal order estimate for this term.

\begin{theorem}
\label{order}
For all $r\in\GR{(}\frac12,r_{\O})$, there exists a positive constant $C$ such
that
\begin{equation*}
\label{cota}
\left\|(T-T_h)f\right\|_{\mQ}
\leq Ch^{\min\left\{\widetilde{r},k\right\}}
\left\|f\right\|_{\mQ}
\qquad\forall f\in\CE
\end{equation*}
and, consequently,
\begin{equation*}
\label{cotgap}
\xi_h\le Ch^{\min\left\{\widetilde{r},k\right\}}.
\end{equation*}
\end{theorem}

\begin{proof}
The proof is identical to that of Lemma~\ref{lemcotste}, but using now
the additional regularity from Lemma~\ref{additional_source}.
\end{proof}

The error estimate for the eigenvalue $\mu\in(0,1)$ of $T$ leads to an
analogous estimate for the approximation of the eigenvalue
$\l=\frac1{\mu}$ of Problem~\ref{P1} by means of the discrete
eigenvalues $\l_h^{(i)}:=\frac1{\mu_h^{(i)}}$, $1\le i\le m$, of
Problem~\ref{P2_disc}. However, the order of convergence in
Theorem~\ref{gap} is not optimal for $\mu$ and, hence, not optimal for
$\l$ either. Our next goal is to improve this order.

\begin{theorem}
\label{cotadoblepandeo} 
For all $r\in \GR{(}\frac12,r_\O)$, there exists a strictly positive constant
$C$ such that
$$\
\left|\l-\l_h^{(i)}\right|
\le Ch^{2\min\{\widetilde{r},k\}}.
$$
\end{theorem}

\begin{proof}

Let $(\bsig_h,u_h)$ be such that $(\l_h,(\bsig_h,u_h))$ is a solution of Problem \ref{P2_disc} with $\|u_h\|_{\mQ}=1$. Also, according to Theorems \ref{gap} and \ref{order}, there exists a solution of Problem \ref{P1} that satisifies
\begin{equation}
\label{eq:error_sigma_u}
\|\bsig-\bsig_h\|_{\bV}+\|u-u_h\|_{\mQ}\leq Ch^{\min\{r,k\}}.
\end{equation}

Let us rewrite Problems \ref{P2}  and \ref{P2_disc} as follows:
\begin{align*}
A\big( (\bsig,u); (\btau,v)\big)&=-\lambda (u,v)_{\mQ}\qquad\,\,\,\,\,\,\forall (\btau,v)\in\bV\times\mQ,\\
A_h\big( (\bsig_h,u_h); (\btau_h,v_h)\big)&=-\lambda (u_h,v_h)_{\mQ}\qquad\forall (\btau_h,v_h)\in\bV_h\times\mQ_h,
\end{align*}
where the bilinear forms $A:\mathcal{V}\times\mQ\rightarrow\mathbb{R}$ and $A:\mathcal{V}_h\times\mQ_h\rightarrow\mathbb{R}$ are defined by
\begin{equation*}
\label{eq:formaA}
A\big( (\bsig,u); (\btau,v)\big):=a(\bsig,\btau)+b(\btau,u)+b(\bsig,v),
\end{equation*}
and 
\begin{equation*}
\label{eq:formaAh}
A_h\big( (\bsig_h,u_h); (\btau_h,v_h)\big):=a_h(\bsig_{h},\btau_h)+b(\btau_h,u_h)+b(\bsig_h,v_h).
\end{equation*}
With these definitions at hand, we have
\begin{align*}
A\big((\bsig-\bsig_h, u-u_h);&(\bsig-\bsig_h,u-u_h)\big)+\lambda(u-u_h,u-u_h)_{\mQ}\\
&=A\big((\bsig_h,u_h);(\bsig_h,u_h)\big)+\lambda(u_h,u_h)_{\mQ}\\
&=A\big((\bsig_h,u_h);(\bsig_h,u_h)\big)+\lambda(u_h,u_h)_{\mQ}+\l_{h}(u_h,u_h)_{\mQ}-\l_{h}(u_h,u_h)_{\mQ}\\
&=A\big((\bsig_h,u_h);(\bsig_h,u_h)\big)+\l_{h}(u_h,u_h)_{\mQ}+\lambda(u_h,u_h)_{\mQ}-\l_{h}(u_h,u_h)_{\mQ}\\
&=A\big((\bsig_h,u_h);(\bsig_h,u_h)\big)-A_{h}\big((\bsig_h,u_h);(\bsig_h,u_h)\big)+(\lambda-\l_{h})(u_h,u_h)_{\mQ}.
\end{align*}
Then, we arrive to the following  identity
\begin{align*}
(\lambda-\lambda_{h})(u_h,u_h)_{\mQ}&=\underbrace{A\big((\bsig-\bsig_h, u-u_h);(\bsig-\bsig_h,u-u_h)\big)+\lambda(u-u_h,u-u_h)_{\mQ}}_{\textbf{I}}\\
&+\underbrace{A_h\left((\bsig_h,u_h);(\bsig_h,u_h)\right)-A((\bsig_h,u_h);(\bsig_h,u_h))}_{\textbf{II}}.
\end{align*}

The aim now is to estimate terms $\textbf{I}$ and $\textbf{II}$. For $\textbf{I}$ we have
\begin{align*}
|\textbf{I}|&=\left|A\big((\bsig-\bsig_h, u-u_h);(\bsig-\bsig_h,u-u_h)\big)-\lambda(u-u_h,u-u_h)_{\mQ} \right|\\
&\leq \left|A\big((\bsig-\bsig_h, u-u_h);(\bsig-\bsig_h,u-u_h)\big)\right|+\left|\lambda(u-u_h,u-u_h)_{\mQ} \right|\\
%&\leq \sum_{\E\in\mathcal{T}_h}\left|a^\E(\bsig-\bsig_h,\bsig-\bsig_h)\right|+|\lambda|\sum_{\E\in\mathcal{T}_h}\|u-u_h\|^2_{0,\E}\\
&\leq \|\bsig-\bsig_h\|^2_{\bV}+\|u-u_h\|^2_{\mQ}\leq Ch^{2\min\{\widetilde{r},k\}}.
\end{align*}

Let $\bsig_h\in\L^2(\O)$ be such that $\bsig_\pi|_{\E}\in\mathbb{P}_k(E)$, for all $\E\in\mathcal{T}_h$. From the definition of $A(\cdot,\cdot)$, $A_h(\cdot,\cdot)$ and $a_h(\cdot,\cdot)$, triangular inequality and  \eqref{20},  the term $\textbf{II}$
is controlled as follows:
\begin{align*}
|\textbf{II}|&= |a_h(\bsig_h,\bsig_h)-a(\bsig_h,\bsig_h)|\\
              &=\left|\sum_{\E\in\mathcal{T}_h}a_h^\E(\bsig_h,\bsig_h)-a^\E(\bsig_h,\bsig_h) \right|\\
               &=\left|\sum_{\E\in\mathcal{T}_h}a_h^\E(\bsig_h-\bpi_h^\E\bsig,\bsig_h-\bpi_h^\E\bsig)-a^\E(\bsig_h-\bpi_h^\E\bsig,\bsig_h-\bpi_h^\E\bsig) \right|\\
             %  &=\left|\sum_{\E\in\mathcal{T}_h}\int_\E\bpi_h^\E\bsig_h\cdot\bpi_h^\E\bsig_h+S^{\E}\big(\bsig_h-\bpi_h^\E\bsig_h,\btau_h-\bpi_h^\E\btau_h\big)-\int_\E\bsig_h\cdot\bsig_h \right|\\
            %  &\leq \left|\sum_{\E\in\mathcal{T}_h}\int_\E\bpi_h^\E\bsig_h\cdot\bpi_h^\E\bsig_h-\int_\E\bsig_h\cdot\bsig_h \right|+\left|\sum_{\E\in\mathcal{T}_h}S^{\E}\big(\bsig_h-\bpi_h^\E\bsig_h,\btau_h-\bpi_h^E\btau_h\big)\right|\\
            %  &\leq \left|\sum_{\E\in\mathcal{T}_h}\left\|\bpi_h^\E\bsig_h\right\|_{0,\E}^2-\|\bsig_h\|_{0,E}^2 \right|+c_1\sum_{\E\in\mathcal{T}_h}a^\E\left(\bsig_h-\bpi_h^\E\bsig_h, \bsig_h-\bpi_h^\E\bsig_h\right)\\
              &\leq C\sum_{\E\in\mathcal{T}_h}\left\|\bsig_h-\bpi_h^\E\bsig\right\|_{0,\E}^2 \leq C\sum_{\E\in\mathcal{T}_h}\left(\left\|\bsig_h-\bsig\right\|_{0,\E}^2+\left\|\bsig-\bpi_h^\E\bsig\right\|_{0,\E}^2\right)\leq Ch^{2\min\{\widetilde{r},k\}},
                %&\leq C\left\|\bsig_h-\bpi_h^\E\bsig_h\right\|_{0,\Omega}^2\leq C\left(\|\bsig_h-\bsig\|_{0,\Omega}^2+\left\|\bsig-\bpi_h^\E\bsig\right\|_{0,\Omega}^2+\left\|\bpi_h^\E(\bsig-\bsig_h) \right\|_{0,\Omega}^2  \right).
\end{align*}
 where we have used \eqref{eq:error_sigma_u} and the properties of the projection. This concludes the proof.

\end{proof}

\setcounter{equation}{0}
\section{Numerical results}
\label{SEC:NUMER}

In this section we report some numerical tests which allows us to assess the performance of the method.
Following the ideas proposed in \cite{BBMR2014}, we have implemented in
a MATLAB code a lowest-order VEM ($k=0$) on arbitrary polygonal meshes.
A natural choice for $S^{\E}(\cdot, \cdot)$ is given by
\begin{equation}
\label{eq:SE}
S^{\E}(\bsig_{h},\btau_{h}):=w_{\E}\sum_{k=1}^{N_{\E}}\left(\int_{e_{k}}\bsig_{h}\cdot \bn\right)\left(\int_{e_{k}}\btau_{h}\cdot \bn\right)
\end{equation}
where $N_{\E}$ represents the number of edges in the polygon $\E$ and  $w_{\E}$ is the so-called stability constant which will be taken of the order of unity, see \cite[Section 5]{BMRR17} for more details.

We report in this section a couple of numerical tests which allowed us
to assess the theoretical results proved above. 
%\subsection{Computation of error estimates for the a priori analysis.}

We begin with some numerical tests to asses the performance of the proposed virtual element method. More precisely, we are interested, first, in the computation of convergence orders to confirm the theoretical results of the analysis. 

With this goal in mind, we present three scenarios in which we will prove our method: the first is to compute the eigenvalues and convergence rates in the unitary square,  the second will correspond to a non-convex domain and the last one considers a square with mixed boundary conditions.
\subsubsection{Test 1:  unit square}
In this test, the domain is the unit square $\O=(0,1)^{2}$. Due the simplicity of this domain, the exact solutions
are known. Indeed, the eigenvalues for this problem are
\begin{equation}
\label{eq:exact1}
\lambda=(m^2+n^2)\pi^2, \qquad m,n\in\mathbb{N}, \,\, m,n\neq 0,
\end{equation}
with the associated eigenfunctions
\begin{equation*}
\label{eq:exact2}
u(x,y)=\sin(m\pi x)\sin(n\pi y), \qquad m,n\in\mathbb{N},\,\, m,n\neq 0.
\end{equation*}

For our numerical tests, we have used four different families of meshes which we describe in the following list:
\begin{itemize}
\item $\CT_{h}^{1}$: triangular meshes;
\item $\CT_{h}^{2}$: square meshes;
\item $\CT_{h}^{3}$:  square meshes with $N$ nodes per side
 then it perturbs all nodes but the central one and the boundary ones
\item $\CT_{h}^{4}$:  trapezoidal meshes which consist of partitions of the domain into $N\times  N$ congruent
trapezoids, all similar to the trapezoid with vertices $(0, 0)$, $(\frac{1}{2},0)$, $(\frac{1}{2},\frac{2}{3})$, $(0, \frac{2}{3})$.
\end{itemize}
The refinement parameter $N$, used to label each mesh, represents the number of elements intersecting each edge.
In the unit square, the eigenfunctions of problem \eqref{eq:laplace}
are smooth and hence, the approximation orders for our method 
are optimal. In Table \ref{tabla:1} we report the first six eigenvalues
computed with our method. In the row 'Order' we present the convergence rates
for each eigenvalue. These order have been computed respect to the exact
eiegenvalues provided by \eqref{eq:exact1}. 

\begin{table}[H]
\begin{center}
\caption{Test 1. The lowest computed eigenvalues $\l_{hi}$, $1\le i\le6$
for different $\CT_{h}$.}
\vspace{0.3cm}
\begin{tabular}{|c|c|c|c|c|c|c|c|}
\hline
$\CT_{h}$& $N$ & $\l_{h1}$ & $\l_{h2}$ & $\l_{h3}$ & 
 $\l_{h4}$ & $\l_{h5}$& $\l_{h6}$  \\
\hline
& 8 &  19.2126  & 46.0629  & 46.2998 &  71.0656  & 86.7093  & 86.9136\\
  & 16 &  19.6065 &  48.5605 &  48.5691 &  76.9474 &  95.6175 &  95.6228\\
$\CT_{h}^{1}$  & 32&  19.7047 &  49.1402 &  49.1409 &  78.4259 &  97.8534 &  97.8615\\
  & 64 &  19.7309 &  49.2955 &  49.2961 &  78.8238 &  98.4859 &  98.4886\\ \hline
  &  Order            &   1.9881  &  1.9825  &  1.9541  &  1.9591  &  1.9372  &  1.9364\\
  \hline
   &Exact        &19.7392 &  49.3480 &  49.3480  & 78.9568  & 98.6960 &  98.6960\\ 
   \hline
\hline
&8  & 18.7724  & 42.0875  & 42.0875 &  65.4027  & 69.7660  & 69.7660\\
  & 16&  19.4886 &  47.2890 &  47.2890 &  75.0894 &  89.3259 &  89.3259\\
$\CT_{h}^{2}$   &32  &  19.6760 &  48.8153 &  48.8153 &  77.9546 &  96.1656 &  96.1656\\
  & 64 &  19.7234 &  49.2137 &  49.2137 &  78.7039 &  98.0505 &  98.0505\\\hline
    &  Order                      &  1.9781 &   1.9218 &   1.9218 &   1.9180 &   1.8346 &   1.8346\\
    \hline
 &Exact    &19.7392 &  49.3480 &  49.3480  & 78.9568  & 98.6960 &  98.6960\\
    \hline
\hline
& 8  & 18.7419 &  41.9133  & 42.0552 &  65.5365 &  69.0829  & 70.2747\\
  & 16 &  19.4846 &  47.2780 &  47.2843 &  74.9461 &  89.3402 &  89.4316\\
$\CT_{h}^{3}$  & 32  &  19.6745 &  48.8119 &  48.8146 &  77.9310 &  96.1702 &  96.1818\\
  & 64 &  19.7230 &  49.2130 &  49.2131 &  78.6981 &  98.0522 &  98.0529\\\hline
  &  Order              &   1.9727 &   1.8968  &  1.8866  &  1.8548 &   1.7757  &  1.7494\\
  \hline
&Exact    &19.7392 &  49.3480 &  49.3480  & 78.9568  & 98.6960 &  98.6960\\
\hline
\hline
& 8  & 18.6654  & 41.8949  & 42.2913  & 64.0705 &  70.1566 &  71.6883\\
  & 16 &  19.4595 &  47.2477 &  47.3714 &  74.6558 &  89.6327 &  90.2307\\
 $\CT_{h}^{4}$ & 32  &  19.6685 &  48.8055 &  48.8384 &  77.8372 &  96.2649 &  96.4358\\
  & 64 &  19.7215 &  49.2112 &  49.2196 &  78.6739 &  98.0769 &  98.1212\\\hline
 &  Order            &    1.9746  &  1.9256 &   1.9295  &  1.9094 &   1.8478  &  1.8567\\
 \hline
&Exact      &19.7392 &  49.3480 &  49.3480  & 78.9568  & 98.6960 &  98.6960\\\hline
\end{tabular}
\label{tabla:1}
\end{center}
\end{table}

It is clear from Table \ref{tabla:1} the double order of convergence for the eigenvalues due the smoothness
of the eigenfunctions.  Also, no spurious eigenvalues are observed in these test, which confirms the 
accuracy and stability of the proposed mixed method. Also In Figure \ref{fig:cuad} we present plots for the first four eigenfunctions. 

\begin{figure}[H]
	\begin{center}
		\begin{minipage}{5cm}
			\centering\includegraphics[height=5cm, width=5cm]{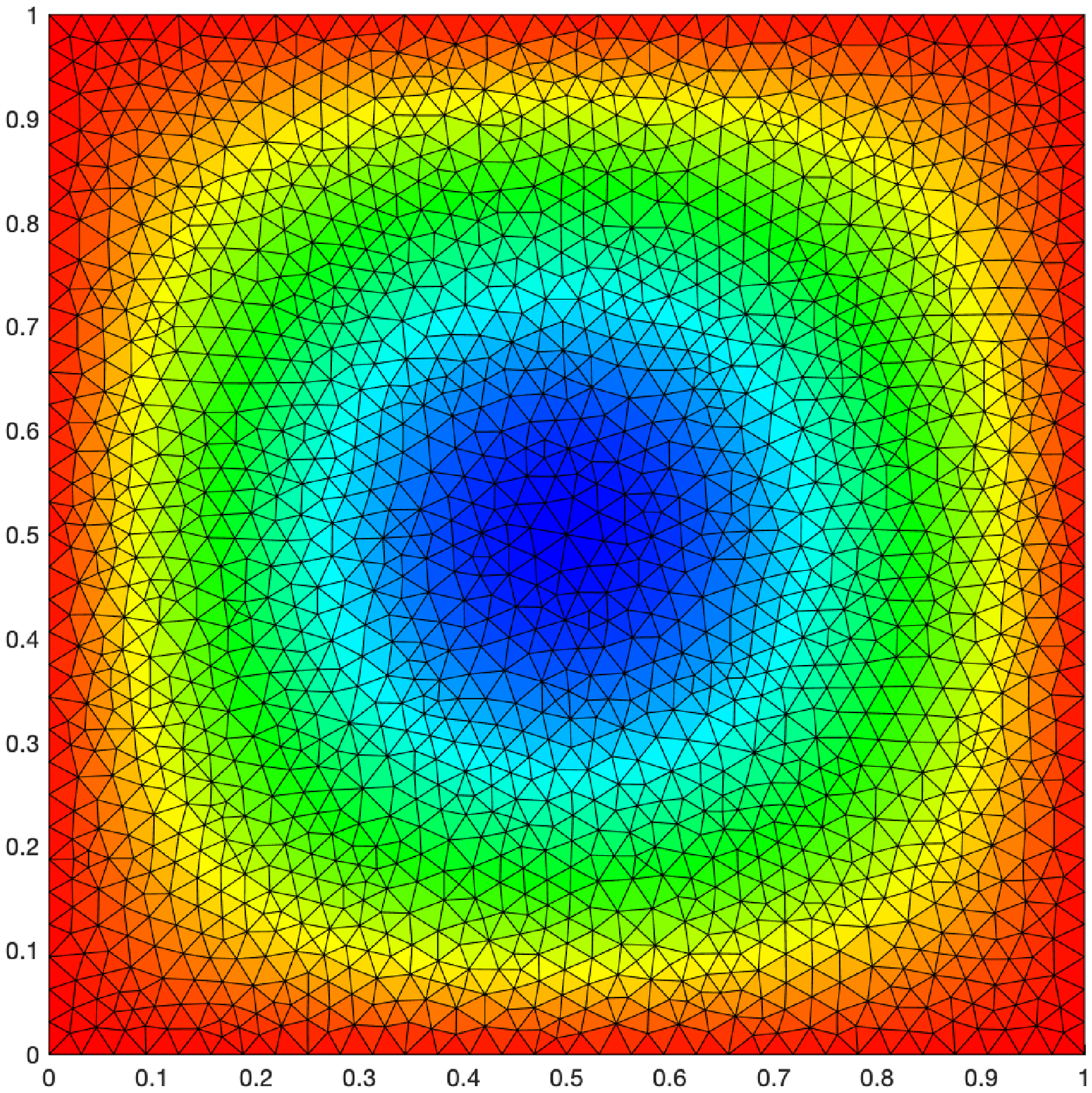}
		\end{minipage}
		\begin{minipage}{5cm}
			\centering\includegraphics[height=5cm, width=5cm]{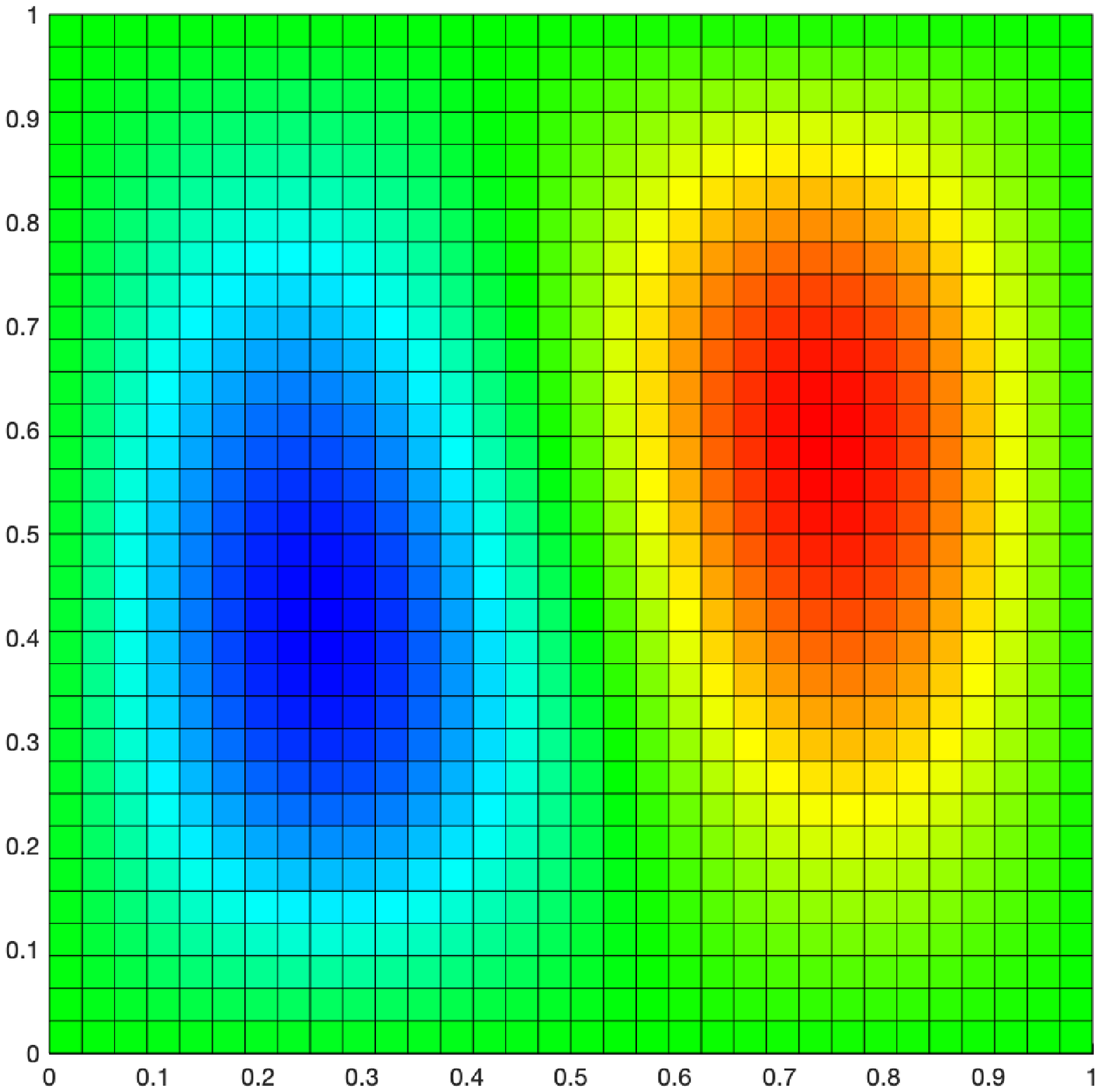}
		\end{minipage}
		\begin{minipage}{5cm}
			\centering\includegraphics[height=5cm, width=5cm]{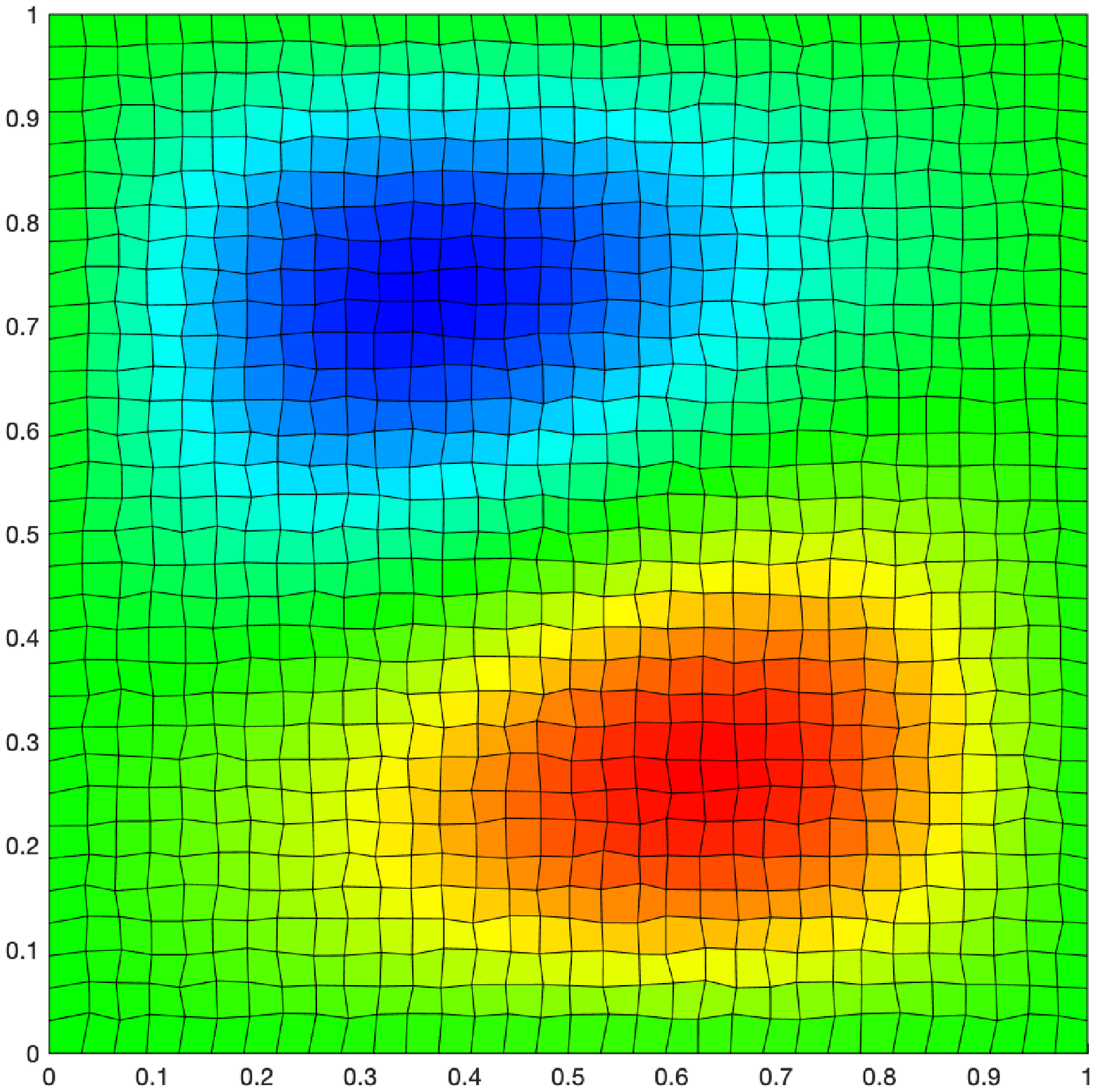}
		\end{minipage}
		\begin{minipage}{5cm}
			\centering\includegraphics[height=5cm, width=5cm]{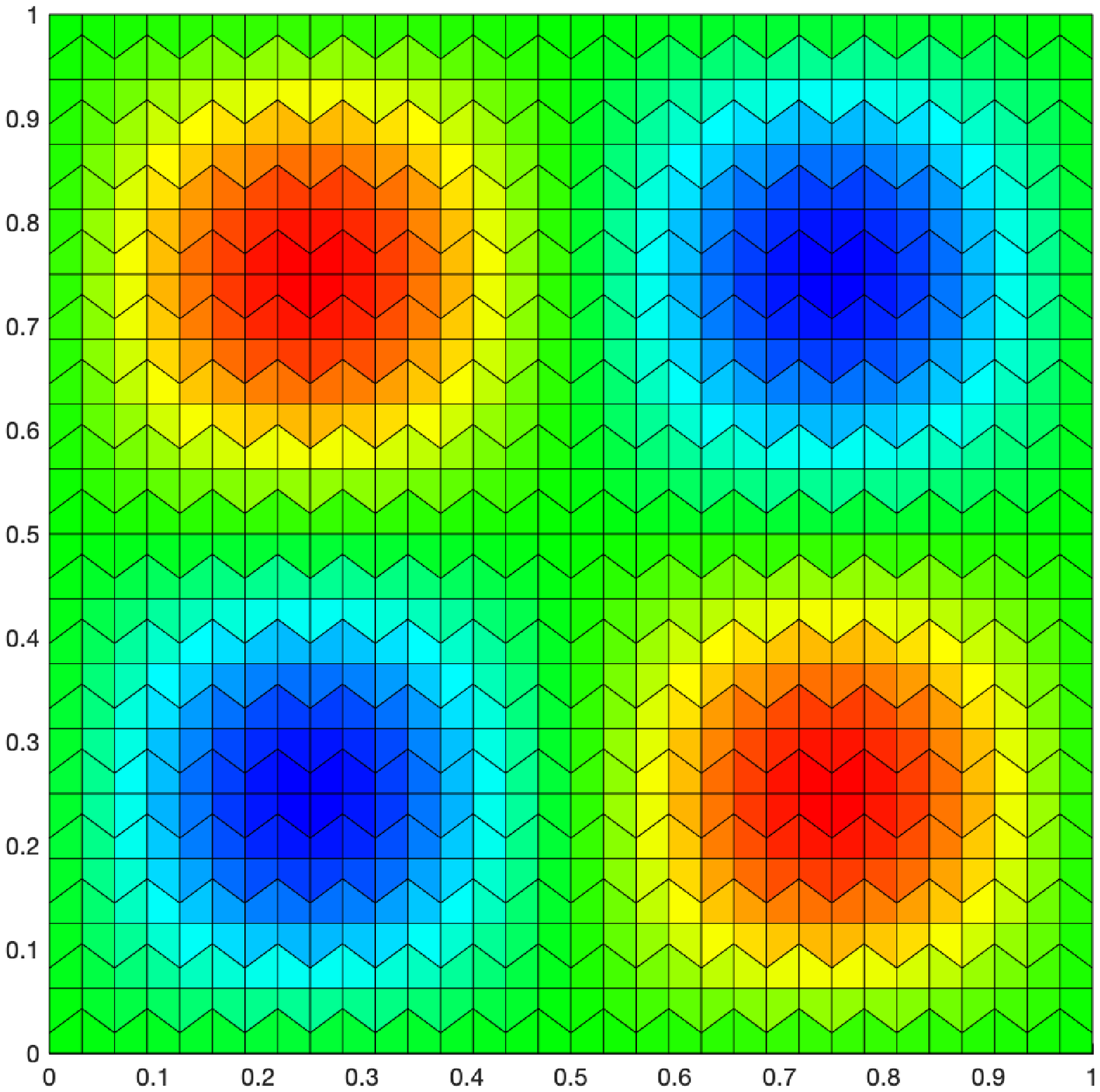}
		\end{minipage}
		\caption{Test 1. Plots of the first four eigenfunctions for$\CT_{h}^{1}$(top left) ,$\CT_{h}^{2}$(top right), $\CT_{h}^{3}$(bottom left) and $\CT_{h}^{4}$(bottom right) with $N=32$}
		\label{fig:cuad}
	\end{center}
\end{figure}
\subsection{Effect of the stability constant}
The aim of this test is to analyze the influence of the stability  (see \eqref{eq:SE}) on the computed spectrum, to know whether the quality of the computations can be affected by this constant.  We will consider the same geometrical configuration of the previous test. We will compute the lowest eigenvalue   for different values $w_{\E}$ using the family of meshes $\CT_{h}^{2}$ .

\begin{table}[H]
\begin{center}
\caption{Test 1. The lowest eigenvalue $\l_{h1}$ for
$w_{\E}=0$ and $w_{\E}=4^{-k}$ with $-6\leq k\leq 6$ and $\CT_{h}^{2}$.}
\begin{tabular}{|c|c|c|c|c|c|c|c|}
\hline
 $N$ & $w_{\E}=4^{6}$ & $w_{\E}=4^{5}$ & $w_{\E}=4^{4}$ & 
 $w_{\E}=4^{3}$ & $w_{\E}=4^{2}$ & $w_{\E}=4^{1}$ & 
 $w_{\E}=4^{0}$ \\
\hline
  8   &0.0623    &  0.2470 &   0.9530  &  3.3404  &  8.9395  & 15.3873  & 18.7724\\
  16  &  0.2469 &   0.9521&    3.3296 &   8.8625 &  15.1606 &  18.4360 &  19.4886\\
   32 &   0.9519&    3.3269 &   8.8434 &  15.1048 &  18.3535 &  19.3965 &  19.6760\\
   64 &   3.3262 &   8.8386  & 15.0909 &  18.3330 &  19.3735 &  19.6524 &  19.7234\\
  128 &  8.8374 &  15.0874  & 18.3279 &  19.3678 &  19.6465 &  19.7174 &  19.7352\\
  256 &  15.0865 &  18.3266 &  19.3664 &  19.6450 &  19.7160 &  19.7338 &  19.7382\\
 \hline
\!\!Order\!\! & 0.3746 &   0.7304 &   1.1463 &   1.5302 &   1.8027  &  1.9400 &   1.9885  \\
\hline
$\l_1$ & 19.7392 & 19.7392 & 19.7392 & 19.7392 & 19.7392 & 19.7392 & 
19.7392 \\
\hline
 $N$ & $w_{\E}=4^{-1}$ & $w_{\E}=4^{-2}$ & $w_{\E}=4^{-3}$ & 
$w_{\E}=4^{-4}$ & $w_{\E}=4^{-5}$ & $w_{\E}=4^{-6}$ & 
$w_{\E}=0$ \\
\hline
 8  &19.8649  & 20.1582  & 20.2328 &  20.2516 &  20.2563 &  20.2575&20.2579 \\
16  & 19.7708  & 19.8427  & 19.8607 &  19.8652 &  19.8664 &  19.8666 & 19.8667\\
32  & 19.7471  & 19.7650 &  19.7695  & 19.7706 &  19.7709 &  19.7709&19.7710 \\
64  & 19.7412  & 19.7457  & 19.7468  & 19.7470  & 19.7471 &  19.7471&19.7471 \\
128 &  19.7397 &  19.7408 &  19.7411  & 19.7412  & 19.7412 &  19.7412& 19.7412\\
256  &19.7393  & 19.7396  & 19.7397  & 19.7397  & 19.7397  & 19.7397&19.7397 \\
\hline
\!\!Order\!\! & 1.9978   & 2.0039 &   2.0049 &   2.0052  &  2.0052 &   2.0052 & 2.0052 \\
\hline
$\l_1$ &19.7392 & 19.7392 & 19.7392 & 19.7392 & 19.7392 & 19.7392 & 19.7392 \\
\hline
\end{tabular}
\label{TABLA:SL}
\end{center}
\end{table}

It can be seen from Table \ref{TABLA:SL}  despite to the fact that the different stabilization numbers do not introduce spurious eigenvalues in the method, the order of convergence is affected  when the stabilization constant is large.

\subsection{Test 2: L-shaped domain}
In this stage, we consider the classic non-convex domain called \emph{L-shaped domain}. 
The non-convexity fo this geometry leads to obtain non-smooth eigenfunctions
due the singularity and hence, non optimal order of convergence. Contrary to the square domain, for this test we do
not have analytical solutions. Hence, we obtain the eigenvalues for different meshes
and different refinement levels and compute the order of convergence respect to an extrapolated value, which
is reported in the row 'Extrap.' of Table \ref{tabla2}. These extrapolated values have been computed by means of a least-square fitting.

In this test we have used four different families of meshes (see Fig. \ref{figL}):
\begin{itemize}
\item $\CT_{h}^{5}$: non-structured hexagonal meshes made of convex hexagons;
\item $\CT_{h}^{6}$: triangular meshes;
\item $\CT_{h}^{7}$:  square meshes;
\end{itemize}

In Table \ref{tabla2} we report the first six eigenvalues of  \eqref{eq:laplace}
computed with our mixed method in the L-shaped domain.

\begin{table}[H]
\begin{center}
\caption{Test 2. The lowest computed eigenvalues $\l_{hi}$, $1\le i\le6$
for different $\CT_{h}$.}
\vspace{0.3cm}
\begin{tabular}{|c|c|c|c|c|c|c|c|}
\hline
$\CT_{h}$& $N$ & $\l_{h1}$ & $\l_{h2}$ & $\l_{h3}$ & 
 $\l_{h4}$ & $\l_{h5}$& $\l_{h6}$  \\
\hline
&18 &  37.7081 &  59.7051  & 76.6486   &113.201&  119.976  & 155.366\\
&38  & 38.3107 &  60.5558  & 78.4827  & 117.017  & 125.714 &  163.376\\
$\CT_{h}^{5} $ &54  & 38.4262 &  60.6812  & 78.7317  & 117.587  & 126.701 &  164.712\\
&70 &  38.4735 &  60.7270  & 78.8254  & 117.797  & 127.079 &  165.205\\
&90  & 38.4999  & 60.7511  & 78.8784  & 117.909  & 127.284  & 165.454 \\
\hline
 & Order   &1.65      &         2.04       &        2.13     &          2.02   &            1.82       &         1.89  \\ \hline
& Extrap. &38.5625  & 60.7943   &78.9530  & 118.109 & 127.721 & 166.000
 \\\hline
$\CT_{h}$& $N$ & $\l_{h1}$ & $\l_{h2}$ & $\l_{h3}$ & 
 $\l_{h4}$ & $\l_{h5}$& $\l_{h6}$  \\
\hline
& 10 &  35.6472 &  56.2124  & 71.8313 &  102.1495 &  105.329 &  133.136\\
& 20  & 37.6435  & 59.5705  & 77.0532  & 113.680 &  121.000 &  155.988\\
$\CT_{h}^{6} $& 30  & 38.0916  & 60.2398  & 78.0999  & 116.088  & 124.487 &  161.234\\
& 50  & 38.3550  & 60.5895   &78.646  & 117.359  & 126.413 &  164.123\\
& 60  & 38.4064  & 60.6502   &78.7409  & 117.580  & 126.763  & 164.641\\
\hline
&  Order   & 1.68         &      1.90       &        1.89        &       1.84           &     1.73       &         1.70
    \\\hline
& Extrap.&38.5495 &   60.8048  & 78.9882 &  118.184  & 127.800 & 166.259 \\
 \hline
 $\CT_{h}$& $N$ & $\l_{h1}$ & $\l_{h2}$ & $\l_{h3}$ & 
 $\l_{h4}$ & $\l_{h5}$& $\l_{h6}$  \\
\hline
&20 &  37.1216  & 58.8887 &  76.4390   &111.187 &  117.648 &  148.894\\
&40 &  38.0961  & 60.2981  & 78.3125  & 116.279  & 124.824  & 161.043\\
$\CT_{h}^{7} $&60 &  38.3168  & 60.5690  & 78.6692  & 117.276  & 126.3007  & 163.633\\
&80  & 38.4047  & 60.6648  & 78.7948  & 117.629  & 126.846  & 164.582\\
&100 & 38.4497  & 60.7094  & 78.8530  & 117.793  & 127.110  & 165.034\\
\hline
& Order  &1.66       &        1.94     &          1.96  &             1.92   &             1.83      &          1.79  \\\hline
& Extrap.  &38.5476  & 60.7943 &  78.9613   &118.111  & 127.635&   166.007
 \\
\hline
\end{tabular}
\label{tabla2}
\end{center}
\end{table}

From Table \ref{tabla2} we observe that the order of convergence for the first eigenvalue is not optimal. This
is expectable due the non-convexity of the chosen domain. In the other hand, the rest of the eigenvalues
converge to the extrapolated ones with double order, since in these cases, the singularity does not 
deteriorate the smoothness of the associated eigenfunctions. 

In Figure \ref{figL} we present plots for the first two eigenfunctions of the Laplace eigenproblem 
in the L-shaped domain, computed with different type of meshes.

\begin{figure}[H]
	\begin{center}
		\begin{minipage}{5cm}
			\centering\includegraphics[height=5cm, width=5cm]{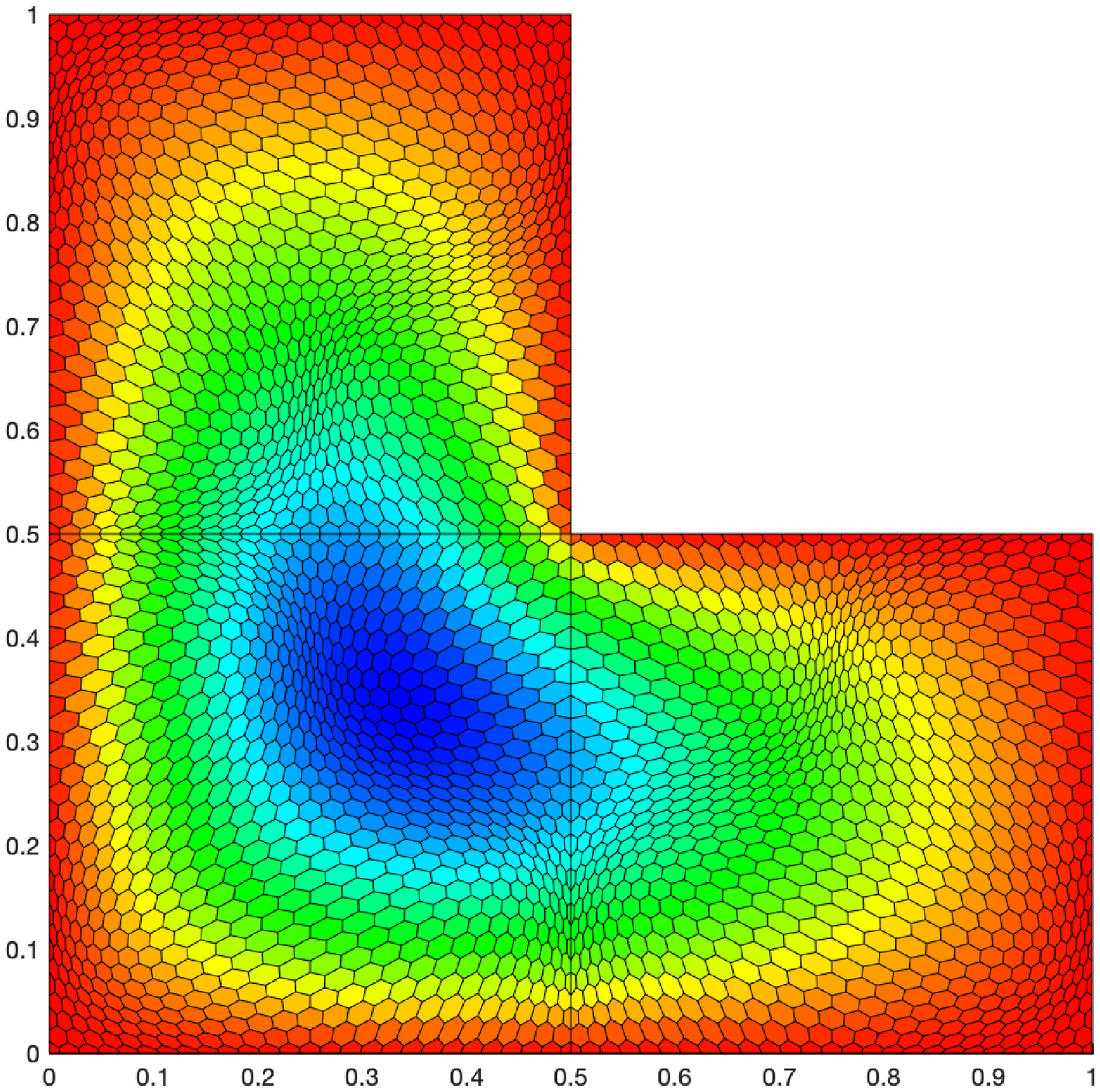}
		\end{minipage}
		\begin{minipage}{5cm}
			\centering\includegraphics[height=5cm, width=5cm]{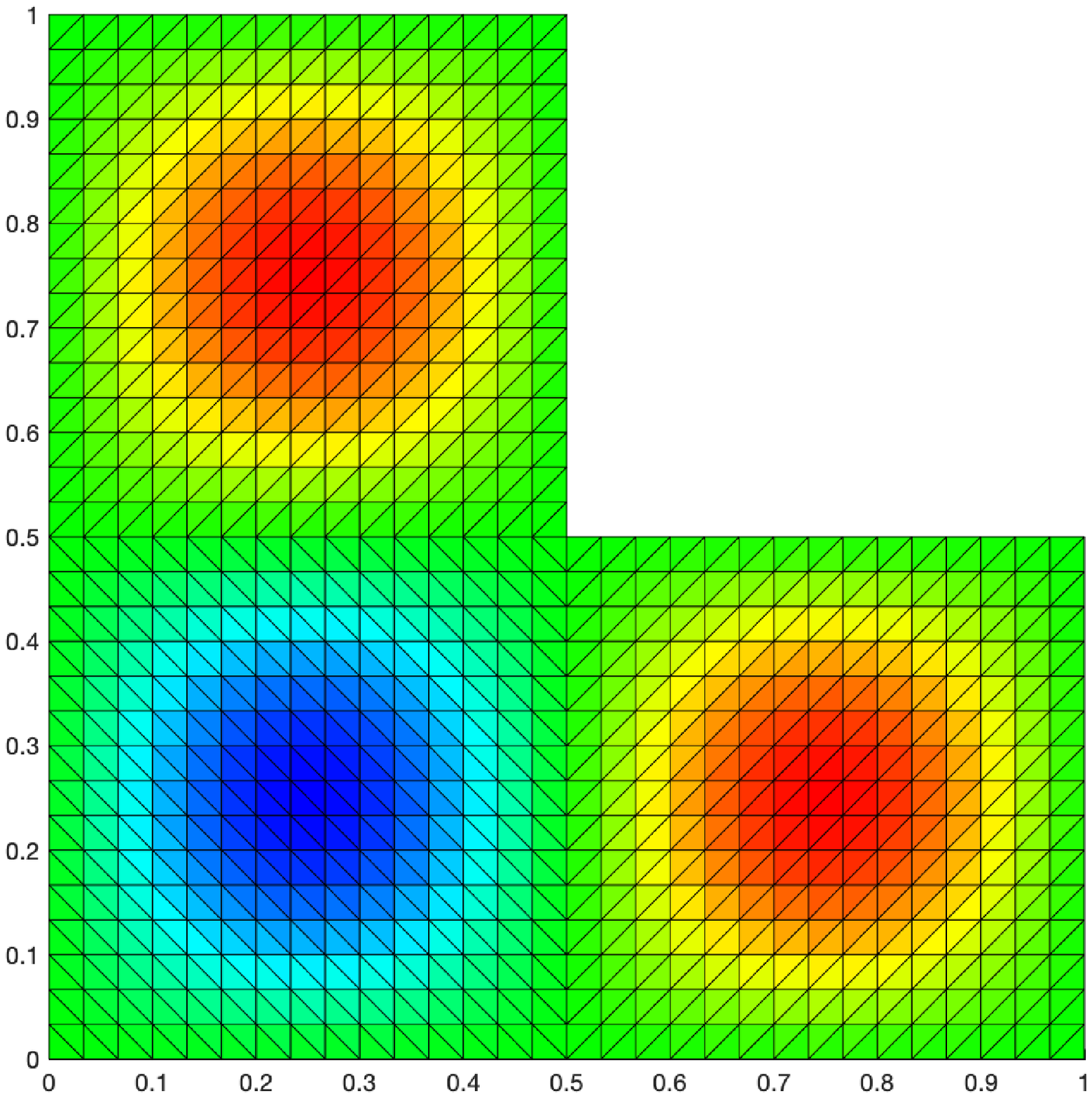}
		\end{minipage}
		%\begin{minipage}{5cm}
		%	\centering\includegraphics[height=5cm, width=5cm]{mallas/m3C.eps}
	%	\end{minipage}
		
\caption{Test 2. Plots of the first two eigenfunctions for $\CT_{h}^{5}$ (left) and $\CT_{h}^{6}$ (right). }
		\label{figL}
	\end{center}
\end{figure}
\subsection{Square with mixed boundary conditions.}
Let us mention the natural extension of system \eqref{eq:laplace} to the mixed boundary conditions case. Let $\O\subset\mathbb{R}^2$ be an open bounded domain with
Lipschitz boundary $\Gamma$. We assume that this boundary is 
splitted in two parts $\Gamma_D$ and $\Gamma_N$ such that
$\Gamma:=\Gamma_D\cup\Gamma_N$. The Laplace eigenvalue problem reads as follows:
\begin{problem}
Find $(\l,u)\in\R\times\HuO$, $u\ne0$, such that
\begin{equation}
\label{eq:laplace_mixed}
\left\{\begin{array}{l}
-\Delta u=\lambda u\quad\text{in }\O,
\\[0.1cm]
u=0  \,\,\,\quad\quad\text{ on }\G_D.\\
u\cdot\boldsymbol{n}=0  \quad\text{ on }\G_N,
\end{array}\right. 
\end{equation}
\end{problem}
where $\boldsymbol{n}$ denotes the outward unitary vector respect to $\Gamma_N$. A mixed variational formulation of \eqref{eq:laplace_mixed} is the following: 
Find $(\l,\bsig, u)\in\R\times\H_0(\div,\O)\times\LO$,
$(\bsig,u)\neq(\boldsymbol{0},0)$, such that 
\begin{equation*}\label{2}
\begin{split}
\int_{\O}\bsig\cdot\btau+\int_{\O}\div\btau u=&0
\qquad\forall\btau\in\H_0(\div,\O),\\
\int_{\O}\div\bsig v=&-\lambda \int_{\O}uv\qquad\forall v\in\LO.
\end{split}
\end{equation*}
where the condition $u\cdot\boldsymbol{n}$ is imposed in the space $\H(\div,\O)$. The analysis of this variational formulations 
is analogous to the Dirichlet boundary condition case considered before. The main difference lies on the regularity of the solution $u$, 
which is affected due the mixed boundary conditions (see  \cite{grisvard}).

If the $\O$-domain is convex,  the analysis of convergence and error estimates holds with no major changes. However, for non-convex domains the analysis is more complex and this will be studied in a future work.

In the present test we consider as computational domain the square $\O:=(-1,1)^2$. The boundary conditions provided in 
\eqref{eq:laplace} are imposed as follows: the condition $u=0$ (i.e., condition on $\Gamma_D$) is considered in the top and bottom of the square, meanwhile the condition on $\Gamma_N$ is considered on the other two sides of the square.

We begin with the analysis of the effects of the stabilization $w_\E$ in the computation of the spectrum. Since we have mixed boundary conditions, which is obviously different compared with our already presented experiments only for the boundary condition $u=0$, it is expectable that spurious modes can arise for certain stabilizations. The following tables report computed eigenvalues for two different meshes: one of triangles ($\CT_h^1$) and the other of squares ($\CT_h^2$). We present this two choices only for simplicity. We mention that for other polygonal meshes the results are similar.

For the experiments, we fix $\CT_h^1$ and $\CT_h^2$ in $N=10$. In tables \ref{TABLA:mix2} and  \ref{TABLA:mix22}  we report the first ten eigenvalues computed with our method, considering the meshes  mentioned before and for different values of $w_\E$.  The numbers in boxes represent spurious eigenvalues and the last column present the exact eigenvalues of the problem, which we compare with the computed ones.

\begin{table}[H]
\begin{center}
\caption{Test 3. Computed lowest eigenvalues for different values of $w_{\E}$ with $N=10$.}
\vspace{0.3cm}
\begin{tabular}{|l|l|l|l|l|l|l|l|l|l|}
\hline
\multicolumn{6}{ |c| }{$\CT_{h}^{1}$} \\
\hline
 $w_{\E}=0$ & $w_{\E}=\frac{1}{100}$ & $w_{\E}=\frac{1}{10}$ & $w_{\E}=\frac{1}{4}$ & 
 $w_{\E}=\frac{1}{2}$ & $\l_{hi}$ \\
\hline
    2.4718   & 2.4716  &  2.4691  &  2.4650 &   2.4582 & 2.4674 \\ 
    4.9502 &   4.9490 &   4.9385 &   4.9212 &   4.8925&  4.9348  \\
    9.9277 &   9.9230  &  9.8810 &   9.8118  &  9.6984  & 9.8701\\
   12.4172 &  12.4101 &  12.3466  & 12.2403 &  12.0572  & 12.3373 \\
   12.4407  & 12.4329 &  12.3634  & 12.2511 &  12.0779 & 12.3374\\
   20.0183  & 19.9999&   19.8354  & 19.5668 &  19.1338 &  19.7404\\
   22.5181  & 22.4946  & 22.2854 &  21.9447 &  21.3974 &22.2094 \\
   25.0231  & 24.9958  & 24.7517&   24.3525 &  23.6593 & 24.6753 \\
   25.1170  & 25.0863  & 24.8131 &  24.3723 &  23.7235 & 24.6757\\
   32.6591  & 32.6124&   32.1980&   31.5263 &  30.4107&   32.0801\\
%   32.8565  & 32.8043&   32.3412 &  31.5986  & 30.4824 & 32.0819 \\
%   40.4910  & 40.4171 &  39.7624 &  38.7122  & 37.0673 &39.4829 \\
%   43.0488  & 42.9664 &  42.2358 &  41.0459 &  39.1429 &  41.9489  \\
%   43.1637  & 43.0764&   42.3077 &  41.1037&   39.2849 &41.9527   \\
%   45.6932  & 45.5986&   44.7636 &  43.4335 &  41.3733 &44.4217 \\
\hline
 $w_{\E}=1$ & $w_{\E}=1.5$ & $w_{\E}=2$ & $w_{\E}=5$ & $w_{\E}=10$  &$\l_{hi}$\\

\hline
 2.4447  &  2.4314  &  2.4182 &   2.3416   & 2.2236&2.4674 \\
    4.8362 &   4.7810 &   4.7270&    4.4256  &  3.9962&  4.9348 \\
    9.4787  &  9.2680   & 9.0657  &  8.0046   & 6.6680&  9.8701\\
   11.6989 &  11.3597 &  11.0385  &  9.4235  &  7.5451& 12.3373\\
   11.7524 &  11.4439  & 11.1505   & 9.6523   & 7.8569& 12.3374\\
   18.3195 &  17.5672  & 16.8702  & 13.5654   &\fbox{10.0645}& 19.7404\\
   20.3747 &  19.4378  & 18.5765  & \fbox{14.5876}   &\fbox{10.5526}& 22.2094\\
   22.3529 &  21.1738  & 20.1051  & \fbox{15.3520}  & \fbox{10.9359}& 24.6753\\
   22.5443 &  21.4719  & 20.4918  & 16.0154  & \fbox{11.5723}&24.6757 \\
   28.2992 &  26.4446  & 24.8059  & \fbox{17.9568}  & \fbox{12.0748}&32.0801\\
%   28.5478 &  26.8312  & 25.2946  & \fbox{18.6293}   &\fbox{12.5246}& 32.0819\\
%   34.1241 &  31.5688  & 29.3315   &\fbox{20.2474}  & \fbox{12.9550}& 39.4829 \\
%   35.7732 &  32.9026  & 30.4332  & \fbox{20.7727} &  \fbox{13.1337}&41.9489\\
%   36.0773 &  33.3198  & 30.9210  &\fbox{21.2019} &  \fbox{13.1862}&41.9527\\
%   37.7572 &  34.6907  & 32.0604  & \fbox{21.8090} &  \fbox{13.5857}& 44.4217\\
\hline
\end{tabular}
\label{TABLA:mix2}
\end{center}
\end{table}
\begin{table}[H]
\begin{center}
\caption{Test 3. Computed lowest eigenvalues for different values of $w_{\E}$ with $N=10$.}
\vspace{0.3cm}
\begin{tabular}{|l|l|l|l|l|l|l|l|l|l|}
\hline
\multicolumn{6}{ |c| }{$\CT_{h}^{2}$} \\
\hline
 $w_{\E}=0$ & $w_{\E}=\frac{1}{100}$ & $w_{\E}=\frac{1}{10}$ & $w_{\E}=\frac{1}{4}$ & 
 $w_{\E}=\frac{1}{2}$  &$\l_{hi}$ \\
\hline
     2.5086  &  2.5073   & 2.4960 &   2.4775  &  2.4472  & 2.4674\\
    5.0171 &   5.0146  &  4.9921   & 4.9550   & 4.8943   &  4.9348\\
   10.5573 &  10.5350 &  10.3390 &  10.0279  &  9.5492 &  9.8701 \\
   13.0658 &  13.0423 &  12.8350 &  12.5054 &  11.9963&  12.3373  \\
   13.0658 &  13.0423&   12.8350&   12.5054 &  11.9963 & 12.3374  \\
   21.1146 &  21.0701 &  20.6780 &  20.0559 &  19.0983  & 19.7404\\
   25.9616 &  25.8275&   24.6801 &  22.9788 &  20.6107 & 22.2094   \\
   28.4702 &  28.3348 &  27.1762 &  25.4563 &  23.0579 & 24.6753  \\
   28.4702 &  28.3348 &  27.1762 &  25.4563 &  23.0579 &  24.6757\\
   36.5189 &  36.3626 &  35.0191 &  33.0067&   30.1599  &  32.0801\\
%   36.5189 &  36.3626 &  35.0191 &  33.0067&   30.1599 & 32.0819 \\
%   \fbox{51.9232} &   \fbox{51.6550} &  47.7458 &  41.7636 &  34.5492 & 39.4829 \\
%   52.7864 &  52.2349&    \fbox{49.3603} &  44.2411 &  36.9963&41.9489    \\
%   55.2950 &  54.7423&   50.2418 &  44.2411&   36.9963&  41.9527   \\
%   55.2950 &  54.7423 &  50.2418 &  45.9576&   41.2215&  44.4217 \\
\hline
 $w_{\E}=1$ & $w_{\E}=1.5$ & $w_{\E}=2$ & $w_{\E}=5$ & $w_{\E}=10$  &$\l_{hi}$\\
\hline
 2.3887    &2.3330   & 2.2798  &  2.0055  &  1.6705 &2.4674  \\
    4.7774 &   4.6660 &   4.5596 &   4.0110 &   3.3409& 4.9348\\
    8.7168  &  8.0179 &   7.4227  &  5.1355  &  3.3930& 9.8701\\
   11.1055 &  10.3509&    9.7025&    7.1410&   \fbox{4.1925}&12.3373\\
   11.1055 &  10.3509&    9.7025  &  7.1410 &   \fbox{4.5674}&12.3374\\
   17.0886 &  14.5946&   12.7359 &    \fbox{7.2193}&    \fbox{4.7619}&19.7404 \\
   17.4335 &  16.0357&   14.8455 &    \fbox{8.4073}&    \fbox{4.8714}&22.2094\\
   19.4774 &  16.9276&   15.0157&    \fbox{ 9.0909} &   \fbox{4.9359}& 24.6753 \\
   19.4774 &  16.9276 &  15.0157&    9.2247&    \fbox{4.9737}& 24.6757\\
   25.6777&   20.4314 &   \fbox{16.9652} &   9.2247&   \fbox{4.9937}&32.0801\\
%   25.8054 &  22.6125 &  19.2450 &    \fbox{9.4986} &  \fbox{5.0000}&32.0819\\
%   25.8054&   22.6125 &  19.2450 &   \fbox{9.7470} &   5.0635&39.4829\\
%   28.0664 &  22.7644 &   \fbox{20.0000} &    \fbox{9.8955} &   5.0635&41.9489\\
%   28.0664 &  22.7644 &   \fbox{20.1586} &    \fbox{9.9750} & \fbox{5.8630}&41.9527\\
%   33.3333 &   \fbox{25.0000} &   \fbox{20.1586} &   \fbox{10.0000}&    \fbox{5.8630}&44.4217\\
\hline
\end{tabular}
\label{TABLA:mix22}
\end{center}
\end{table}
We observe from tables  \ref{TABLA:mix2} and  \ref{TABLA:mix22}  that the behaviour of the spurious eigenvalues clearly depend on the choice of $w_\E$. Notice  that when $w_\E$ is close to zero (even for the case $w_\E=0$) the spurious vanish, meanwhile when $w_\E$ increases, the spurious eigenvalues appear, for both $\CT_h^1$ and $\CT_h^2$, from $w_\E=5$ in forward. Moreover, we observe that the method do not introduce spurious for $w_\E<2$. This phenomenon is analogous for other meshes. 

We remark that  the appearance of spurious eigenvalues is not visible for the domain with the Dirichlet boundary condition, which leads to conclude that the mixed boundary is also a relevant factor at the moment to compute the spectrum of the Laplace eigenvalue problem.

Now we are interested to observe if the refinement of the meshes influence the appearance  of spurious eigenvalues when we choose  a value of $w_\E$ that introduces pollution of the spectrum. To do this task, we chose
$w_\E=10$ (since with this value of $w_E$ several spurious eigenvalues are observed) and refine the meshes $\CT_h^1$ and $\CT_h^2$ that we have used in the previous tests.

\begin{table}[H]
\begin{center}
\caption{Test 3. Computed lowest eigenvalues for $w_{\E}=10$.}
\vspace{0.3cm}
\begin{tabular}{|l|l|l|l|l|l|l|l|l|l|}
\hline
\multicolumn{5}{ |c| }{$\CT_{h}^{1}$} \\
\hline
$\l_{hi}$& $N=10$ & $N=20$ & $N=30$ & $N=40$ \\
\hline
2.4674 & 2.2236    &2.4071   & 2.4399 &   2.4520\\
  4.9348  &    3.9962 &   4.6921   & 4.8230&    4.8729\\
 9.8701&    6.6680  &  8.9636   & 9.4461  &  9.6299\\
12.3373&    7.5451  & 10.9239  & 11.6687   &11.9544\\
12.3373&    7.8569  & 10.9440  & 11.6810  & 11.9617\\
19.7404&   \fbox{10.0645} &  16.4216  & 18.0867 &  18.8010\\
 22.2094&   \fbox{10.5526} &  18.0007  & 20.1514  & 21.0296\\
 24.6753&   \fbox{10.9359} &  19.5387  & 22.0980  & 23.1650\\
 24.6757&   \fbox{11.5723} &  19.8019  & 22.1888  & 23.2291\\
 32.0801&   \fbox{12.0748} &  23.9208  & 27.8445  & 29.5927\\
%32.0819&   \fbox{12.5246} &  24.1099  & 28.0099  & 29.6691\\
% 39.4829&   \fbox{12.9550} &  28.0105  & 33.3759  & 35.9100\\
%41.9489&   \fbox{13.1337} & \fbox{29.0138}  & 35.0916  & 37.8068\\
%41.9527&   \fbox{13.1862} & \fbox{29.3178} & 35.1181  & 37.9604\\
% 44.4217&   \fbox{13.5857} & \fbox{29.9505}  &36.7696  & 39.8457\\
   \hline
\hline
\multicolumn{5}{ |c| }{$\CT_{h}^{2}$} \\
\hline
$\l_{hi}$& $N=10$ & $N=20$ & $N=30$ & $N=40$ \\
\hline
2.4674 &1.6705  &  2.2045 &   2.3432 &   2.3960\\
  4.9348 &   3.3409&    4.4090  &  4.6864&    4.7919\\
 9.8701 &   3.3930  &  6.6819  &  8.1431  &  8.8180\\
12.3373 &  \fbox{4.1925}  &  8.8864 &  10.4863 &  11.2139\\
12.3373 &   \fbox{4.5674}  &  8.8864  & 10.4863&   11.2139\\
19.7404 &   \fbox{4.7619} &  10.7096&   15.0342&   17.5084\\
 22.2094&   \fbox{4.8714} &  12.9141 &  16.2862&   17.6359\\
24.6753 &   \fbox{4.9359} &  12.9141 &  17.3774&   19.9043\\
 24.6757 &  \fbox{ 4.9737} &  13.3637 &  17.3774 &  19.9043\\
  32.0801&   \fbox{4.9937} &  \fbox{13.5721} &  21.3607 &  26.3263\\
%32.0819 &   \fbox{5.0000} &  \fbox{15.4868} &  23.1773 &  26.3263\\
% 39.4829 &   5.0635 & \fbox{15.7766} &  23.1773&   26.7275\\
%41.9489 &   5.0635  & \fbox{15.7766} &  23.7039&   29.1235\\
%41.9527&   \fbox{5.8630} &  \fbox{16.7702}&   23.7039&   29.1235\\
% 44.4217 &  \fbox{5.8630}  & \fbox{17.3915} &  \fbox{26.5266}&   35.0167\\
\hline
\end{tabular}
\label{TABLA:mix3}
\end{center}
\end{table}
      
    From table \ref{TABLA:mix3} we observe that the spurious eigenvalues disappear when the meshes are refined, which is expectable, since the accuracy of the VEM is improved when the meshes are refined. Similar results have been obtained for other numerical methods that depend on some particular stabilization parameter (for instance, the DG method \cite{LM,MR3962898}).

      In table \ref{tabla:mix1} we report the first six eigenvalues computed with our method considering different polygonal meshes and $w_\E=1$. It is clear that the double order of convergence is obtained with this configuration of the geometry, as is expected.
\begin{table}[H]
\begin{center}
\caption{Test 3. The lowest computed eigenvalues $\l_{hi}$, $1\le i\le6$
for different $\CT_{h}$.}
\vspace{0.3cm}
\begin{tabular}{|c|c|c|c|c|c|c|c|}
\hline
$\CT_{h}$& $N$ & $\l_{h1}$ & $\l_{h2}$ & $\l_{h3}$ & 
 $\l_{h4}$ & $\l_{h5}$& $\l_{h6}$  \\
\hline
 &8 &  2.4343 &   4.8005   & 9.3442  & 11.4928  & 11.5519 &  17.7203\\
   &16 &   2.4592    &4.9004   & 9.7437 &  12.1144  & 12.1295  & 19.2004\\
$\CT_{h}^{1}$    &32  &  2.4653   & 4.9263  &  9.8362 &  12.2820  & 12.2839 &  19.6013\\
   &64   & 2.4669  &  4.9327&    9.8613  & 12.3237 &  12.3239&   19.7054\\ \hline
      &   Order &   2.0200   & 1.9600 &   2.0800  &  1.9100 &   1.9100 &   1.8900\\\hline
      &   Extr.  &  2.4673&    4.9351 &   9.8670 &  12.3406 &  12.3391 &  19.7470\\   \hline
\hline
&8  & 2.3465 &   4.6931  &  8.1753 &  10.5219  & 10.5219  & 16.3507\\
  & 16&    2.4361 &   4.8722&    9.3862  & 11.8223  & 11.8223&   18.7724\\
  $\CT_{h}^{2}$  &32  &  2.4595  &  4.9190  &  9.7443&   12.2038 &  12.2038&   19.4886\\
  & 64   & 2.4654  &  4.9308    &9.8380 &  12.3034 &  12.3034&   19.6760\\\hline
   &Order   & 1.9400   & 1.9400   & 1.7900&    1.8000  &  1.8000   & 1.7900\\\hline
  &Extr.       &    2.4676 &   4.9353 &   9.8822   &12.3496&   12.3496 &  19.7644\\    \hline
\hline
& 8&   2.3474 &   4.6887&    8.1662  & 10.4683&   10.5147&   16.3106\\
 &  16 &  2.4363  &  4.8705 &   9.3909 &  11.8172 &  11.8204&   18.7349\\
$\CT_{h}^{3}$ &  32  &  2.4596 &   4.9186 &   9.7456 &  12.2024&  12.2029 &  19.4837\\
 &  64 &  2.4654 &   4.9308  &  9.8383 &  12.3032&   12.3032 &  19.6746\\\hline
&  Order  &  1.9500   & 1.9300  &  1.8100  &  1.8300&    1.8000   & 1.7400\\\hline
&         Extr.  &  2.4675  &  4.9354  &  9.8811  & 12.3473  & 12.3495 &  19.7787\\\hline
\hline
& 8  & 2.3575  &  4.6664  &  8.3096 &  10.4737   &10.5728 &  16.0176\\
&   16  &  2.4390  &  4.8649 &   9.4302   &11.8119  & 11.8429 &  18.6639\\
$\CT_{h}^{4}$ &  32 &   2.4603 &   4.9171 &   9.7561&   12.2014 &  12.2096 &  19.4593\\
  & 64   & 2.4656 &   4.9304 &   9.8410 &  12.3028 &  12.3049 &  19.6685\\\hline
  &      Order  &  1.9500   & 1.9300  &  1.8100 &   1.8100 &   1.8200  &  1.7700\\\hline
   &     Extr.   & 2.4676 &   4.9355 &   9.8799 &  12.3491 &  12.3477 &  19.7701\\\hline
     \end{tabular}
\label{tabla:mix1}
\end{center}
\end{table}
                        
Finally, in figure \ref{FIG:MPMix} we present plots of  the first four eigenfunctions for the Laplace eigenvalue problem with mixed boundary conditions, obtained with different polygonal meshes.
\begin{figure}[H]
\begin{center}
\begin{minipage}{5cm}
\centering\includegraphics[height=5cm, width=5cm]{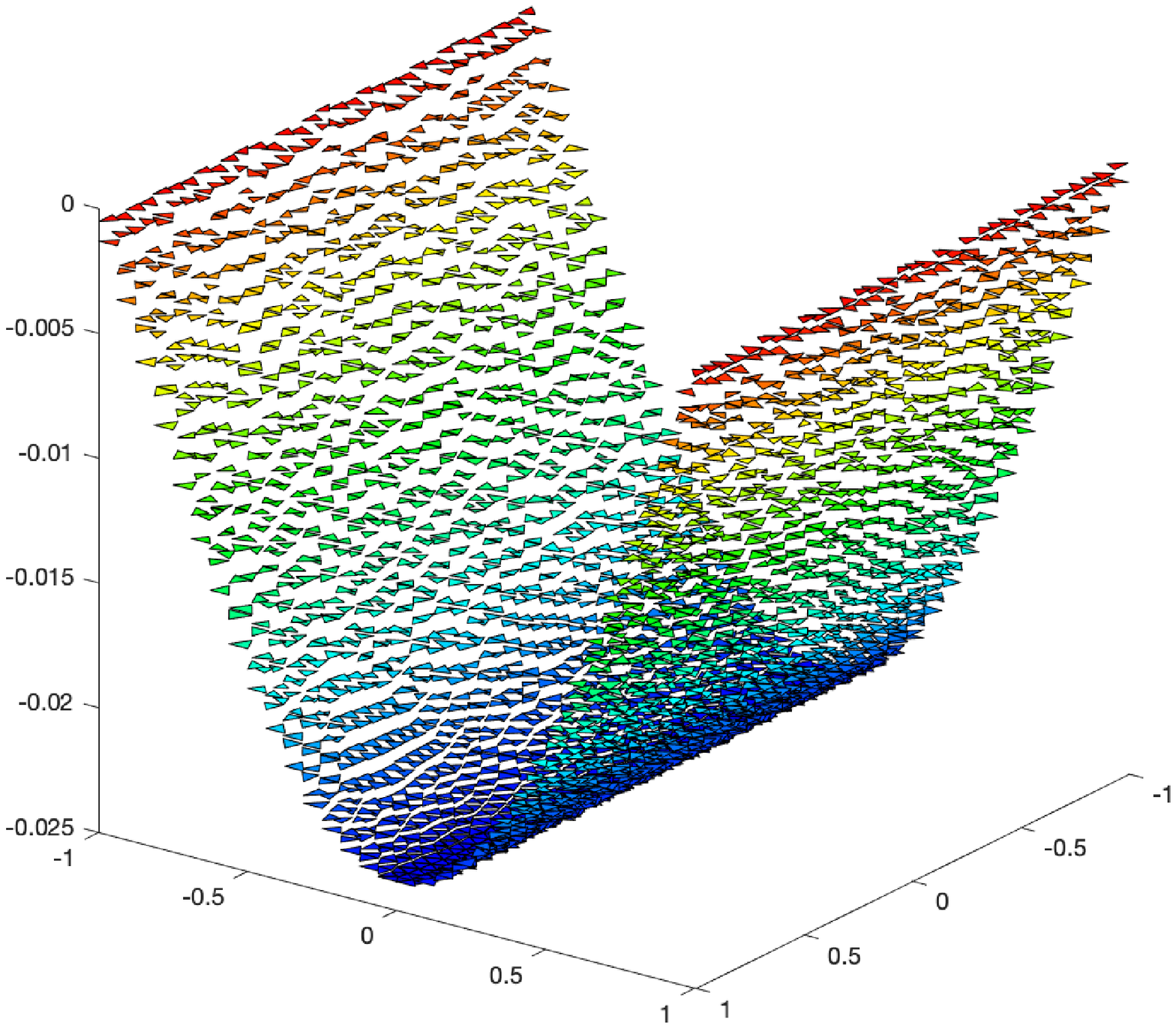}
\end{minipage}
\begin{minipage}{5.cm}
\centering\includegraphics[height=5cm, width=5cm]{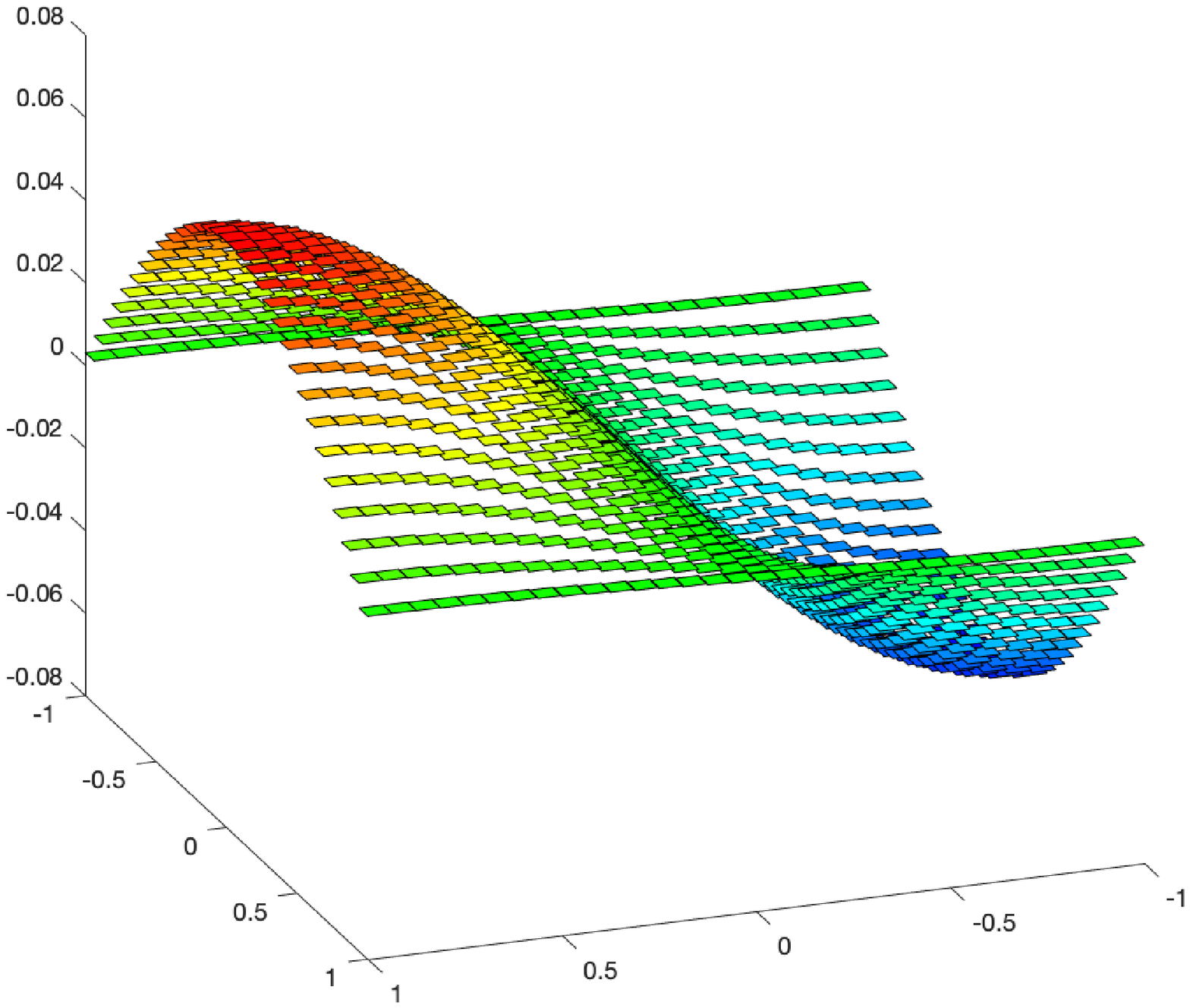}
\end{minipage}
\begin{minipage}{5cm}
\centering\includegraphics[height=5cm, width=5cm]{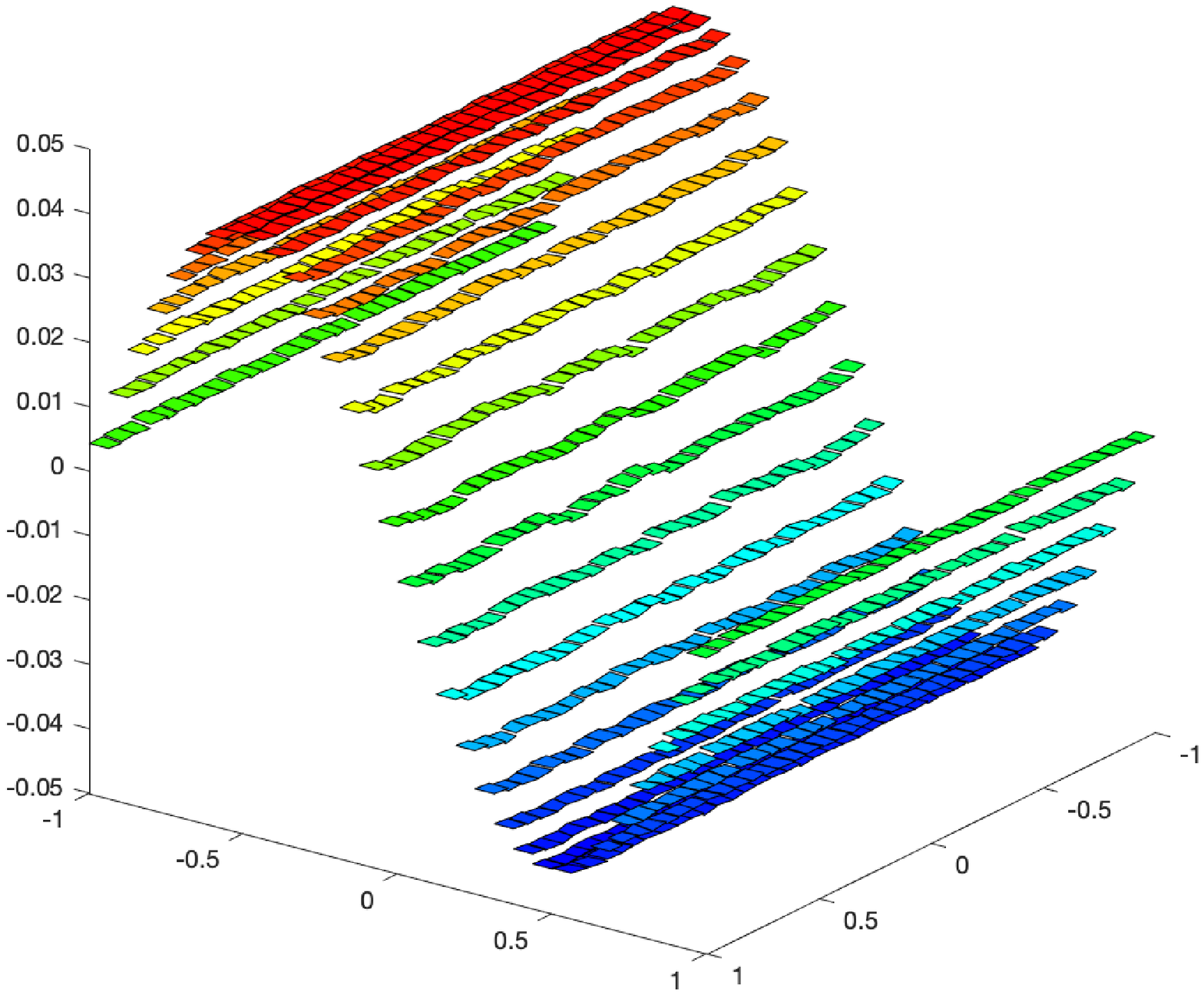}
\end{minipage}
\begin{minipage}{5cm}
\centering\includegraphics[height=5cm, width=5cm]{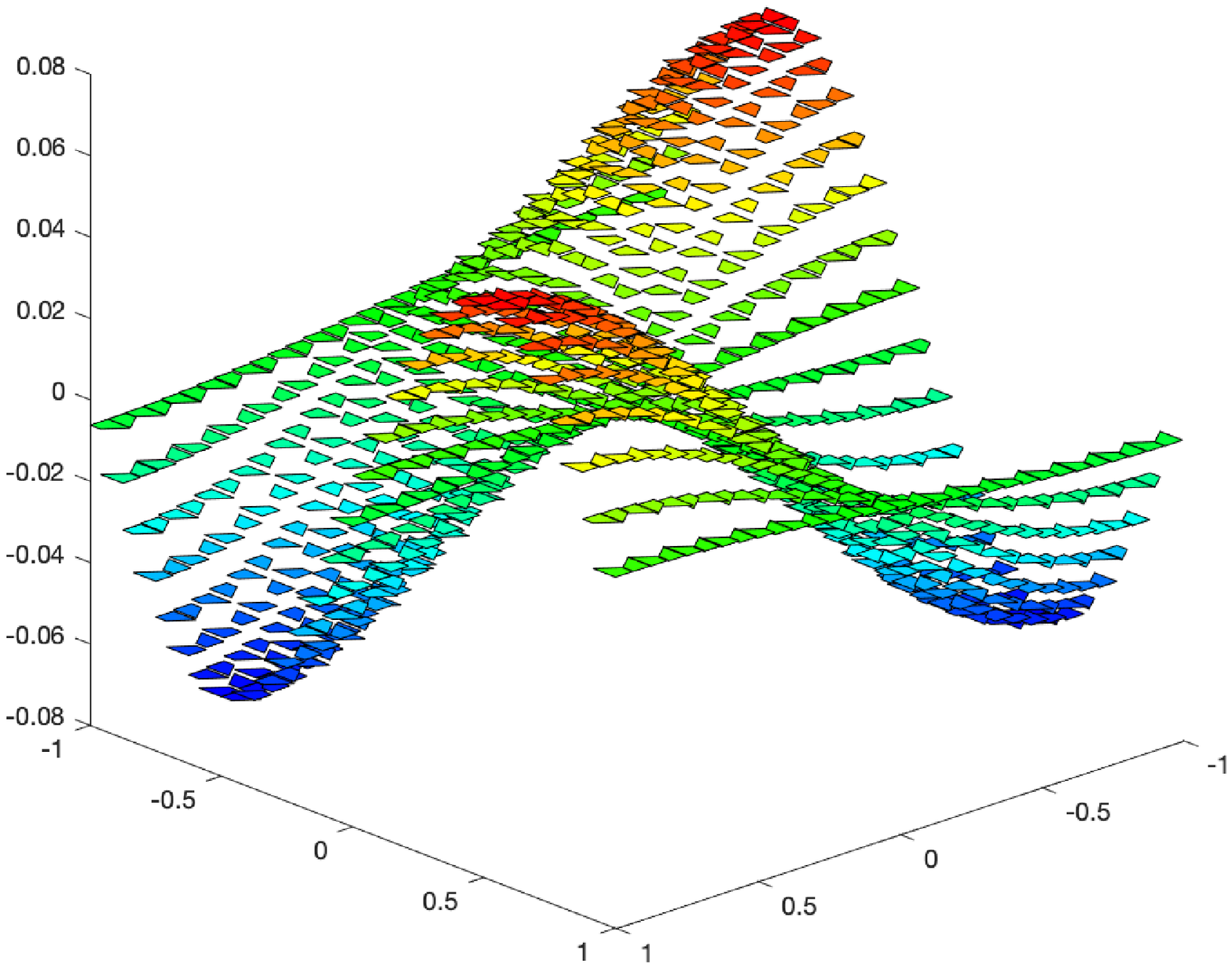}
\end{minipage}
\caption{Test 3. Plots of the first four   eigenfunctions for $\CT_{h}^{1}$ (top left), $\CT_{h}^{2}$ (top right), $\CT_{h}^{3}$ (bottom left) and $\CT_{h}^{4}$ (bottom right) computed with different meshes and $N=32$.}
\label{FIG:MPMix}
\end{center}
\end{figure}

%\begin{table}[H]
%\begin{center}
%\caption{Test 1. Computed lowest eigenvalues $\l_{hi}$, $1\le
%i\le4$, on different meshes with $\sigma_E=1$.}
%\begin{tabular}{|c|c|c|c|c|c|c|c|c|}
%\hline
%$\CT_h$   & $\l_{hi}$ & $N=8$ & $N=16$ & $N=32$ & $N=64$ & 
%Order & Extr&\cite{Lin} \\
%\hline
%          & $\l_{h1}$ &0.23986  &0.24000  & 0.24006 &  0.24007 & 
%	  1.99  & 0.24008 &0.24008\\
%$\CT_h^1$ & $\l_{h2}$ &  1.49536  & 1.49305  & 1.49250 &  1.49235 & 
%2.06 & 1.49232 &1.49230\\
%          & $\l_{h3}$ &1.49540  & 1.49306 &  1.49250 &  1.49235 & 
%	  2.04  &  1.49232&1.49230 \\
%          & $\l_{h4}$ & 2.09016 &  2.08446  & 2.08313 &  2.08277 & 
%	  2.07  &  2.08269 &2.08265 \\
%\hline
%          & $\l_{h1}$ &0.23932 &  0.23989 &  0.24003 &  0.24007 & 
%	  2.00  & 0.24008 &0.24008  \\
%$\CT_h^2$ & $\l_{h2}$ &1.47863  & 1.48893 &  1.49146 &  1.49209 & 
%2.02  & 1.49230 &1.49230 \\
%          & $\l_{h3}$ & 1.47863 &  1.48893 &  1.49146 &  1.49209 & 
%	  2.02  & 1.49230&1.49230 \\
%          & $\l_{h4}$ &  2.08108  & 2.08226  & 2.08255  & 2.08262 & 
%	  2.00  & 2.08265 &2.08265\\
%\hline
%\end{tabular}
%\label{TABLA:1}
%\end{center}
%\end{table}
  
%   \subsection{A posteriori tests.}     

%-----------------------------------------------------------------------
\section*{Acknowledgments}
The authors are deeply grateful to Prof. David Mora (Universidad del B\'io-B\'io, Chile) for the fruitful discussions and  comments. 

FL was partially supported by
CONICYT-Chile through FONDECYT Postdoctorado project 3190204 (Chile). GR  was supported by
CONICYT-Chile through FONDECYT project 11170534 (Chile). %-----------------------------------------------------------------------

\bibliographystyle{amsplain}

\end{document}